\documentclass[a4paper]{amsart}

\usepackage{amsmath,amsthm,amssymb,latexsym,epic,bbm,comment, mathrsfs}

\usepackage{graphicx,enumerate,stmaryrd,color}

\usepackage[all]{xy}

\usepackage[active]{srcltx}

\usepackage[parfill]{parskip}

\usepackage{enumerate}

\newtheorem{theorem}{Theorem}
\newtheorem*{thm}{Theorem}

\newtheorem{lemma}[theorem]{Lemma}

\newtheorem{corollary}[theorem]{Corollary}

\newtheorem{proposition}[theorem]{Proposition}

\newtheorem{conjecture}[theorem]{Conjecture}

\newtheorem{example}[theorem]{Example}

\newtheorem{remark}[theorem]{Remark}

\theoremstyle{definition}
\newtheorem{definition}[theorem]{Definition}

\title{$J$-equivalence for associative algebras}
\author{Helena Jonsson}

\begin{document}

\begin{abstract}
We study the combinatorics of an analogue of 
Green's $\mathcal{J}$-relation (a.k.a. the two-sided relation) 
for the bicategory of finite-dimensional bimodules over finite-dimensional
associative algebras over a fixed field. In particular,
we provide a number of exameples for both
$\mathcal{J}$-equivalent and  inequivalent but $\mathcal{J}$-comparabe algebras.
\end{abstract}

\maketitle

\section{Introduction}

If one wants to understand the structure of some semigroup, say $S$, the first question one asks
is how the Green's relations for $S$ look like. Green's relations $\mathcal{L}$,
$\mathcal{R}$ and $\mathcal{J}$ are equivalence relations on $S$ that describe equivalence 
classes of principal left, right and two-sided ideals, see \cite{Gr}.

More generally, analogues of Green's relations can be similarly defined in the setup
of additive bicategories, see \cite{MM1,MM2}, or multisemigroups, see \cite{KM}.
One natural example of an additive bicategory appearing in representation theory
is that of finite-dimensional bimodules over finite-dimensional
associative algebras over a fixed field  $\Bbbk$. This bicategory,
which we denote by $\mathscr{BM}_\Bbbk$, has, as objects,
all  finite-dimensional associative $\Bbbk$-algebras. The $1$-morphisms
are finite-dimensional bimodules, the $2$-morphisms are bimodule homomorphism
and the horizontal composition is given by tensor product.

The aim of this paper is to make some first steps in understanding the 
combinatorics of the $\mathcal{J}$-relations in this bicategory, that is for
tensor product of bimodules over associative algebras. Our interest in this object
stems from its relevance and importance in $2$-representation theory, see, for example,
\cite{MM1,MMMTZ} and references therein. In particular, for a natural class of 
bicategories, the corresponding $\mathcal{J}$-equivalence classes serve as a
crude invariant, called the \emph{apex}, see \cite{CM}, that is used in 
classification of ``simple'' birepresenatations. Therefore understanding
the structure of the $\mathcal{J}$-relation is very useful.

For example, in \cite{Jo3} and \cite{JS}, the author, together with M.~Stroi{\'n}ski, 
classified ``simple'' birepresentations for all but one apex in the case of
the bicategory of bimodules over certain radical square zero Nakayama algebras,
which include the algebra of dual numbers. These results rely heavily on the 
explicit description of the cell combinatorics of such bimodules, provided in \cite{Jo2}.
The fact that such an explicit description is at all possible is due to tameness 
of the category of bimodules over these radical square zero Nakayama algebras.

Unfortunately, for most algebras, the category of finite-dimensional bimodules 
is wild, meaning that, at the present stage, it is impossible to classify all
indecomposable objects. Indeed, it was proven in \cite[Theorems~7.1,7.2]{LS2} that, 
up to Morita equivalence and direct sums, the only tame (including finite representation type) 
bimodule categories appear for radical square zero Nakayama algebras, along with 
five additional algebras. This means that, in most cases, the brute force 
approach of \cite{Jo2} to understand the $\mathcal{J}$-relation is not possible.
Therefore the problem to understand the $\mathcal{J}$-relation in this setup
is rather non-trivial.

In this paper we make some first steps in understanding the combinatorics of 
the $\mathcal{J}$-relation for the bicategory $\mathscr{BM}_\Bbbk$.

For example, we prove the following general result, see Theorems \ref{thm_group} and \ref{thm_skew}.
\begin{thm}
Let $\Bbbk$ be a field, $A$ a $\Bbbk$-algebra, and $G$ a finite group acting on $A$ via suitable automorphisms. Assume further that the characteristic of $\Bbbk$ does not divide the order of $G$.
Then $A$ is $J$-equivalent both to the skew-group algebra $A\ast G$ and the subalgebra of invariants $A^G$.
\end{thm}

Here is one of many explicit examples of $J$-(in)equivalence we prove in Section~\ref{s_examples}.
\begin{thm}
Let $\Bbbk$ be an algebraically closed field. Denote by $A_n$ the path algebra of the uniformly oriented Dynkin diagram of type $A_n$ modulo the square of its radical. Then the following chain of strict $J$-inequalities hold.
\begin{displaymath}
A_2>_J A_3>_J A_4>_J \ldots >_J \Bbbk [x]/(x^2)
\end{displaymath}
\end{thm}

This paper is organized as follows: in Section~2, we provide relevant background, state the definition of $J$-equivalence, and recall some related equivalences of algebras - in particular separable equivalence in the sense of \cite{Ka1}, \cite{Li}.
In Section~3, we consider more carefully the bimodules inducing a $J$-relation.
In Section~4 we prove that taking factor algebras induces $J$-relation, and consider split and separable algebra extensions. Further we consider two constructions involving group actions on an algebra: the invariant subalgebra and the skew group algebra.
In Section~5 we generalize the work of S.~Peacock on separable equivalence to $J$-equivalence, and prove that it preserves representation type and is stable under forming tensor algebras.
Section~6 is devoted to examples of $J$-equivalent and -comparable algebras.
In Section~7, we study left-right projective bimodules over certain classes of algebras. As a byproduct, we show that many of our examples of $J$-comparable algebras are not related in the sense of separable equivalence.
Finally, in Section~8 we look towards possible future directions, and conjecture that $J$-equivalent algebras have the same Loewy length.

\section{Preliminaries}

\subsection{Notation and setup}
Throughout, $\Bbbk$ is a field. In some sections we add requirements 
about characteristic of $\Bbbk$ and/or it being algebraically closed.
Algebra means finite-dimensional, associative, unital $\Bbbk$-algebra, 
and all (bi)modules are finite-dimensional unless otherwise stated.

Let $A$ and $B$ be $\Bbbk$-algebras. We denote by $A$-mod, mod-$A$ and 
$A$-mod-$B$ the categories of finite-dimensional left $A$-modules, 
right $A$-modules, and $A$-$B$-bimodules, respectively. For an 
$A$-$B$-bimodule $M$ we sometimes write ${}_AM_B$, ${}_AM$ and $M_B$ to 
emphasize which module structure we consider. The notions of $A$-$B$-bimodules 
and left $A\otimes_\Bbbk B^\mathrm{op}$-modules will be used interchangeably.

We compose maps from right to left.

\subsection{$\Bbbk$-split and (left-right) projective bimodules}

An $A$-$B$-bimodule ${}_AM_B$ is called {\em left projective} if the 
left $A$-module ${}_AM$ is projective. Similarly, $M$  is called 
{\em right projective} if the right $B$-module $M_B$ is projective. 
If both ${}_AM$ and $M_B$ are projective, we say that $M$ is 
{\em left-right projective}. We reserve the term {\em projective bimodule} 
for  the case when $M$ is projective as a bimodule. If $M$ is projective as 
a bimodule, then $M$ is left-right projective, but the converse is not 
true in general.  For example, the regular $A$-$A$-bimodule ${}_AA_A$ is 
left-right projective for any $A$, but projective if and only if 
$A$ is separable (and thus semisimple).

Furthermore, if ${}_AM_B$ and ${}_BN_C$ are left-right projective, then ${}_AM\otimes_BN_C$ is again left-right projective; see e.g. \cite[Remark~2.2]{ZZ} for an explicit argument. 
It therefore follows that the full subcategory $A\text{-lrproj-}A$ of $A\text{-mod-}A$, consisting of left-right projective bimodules, is a monoidal subcategory of the monoidal category $A\text{-mod-}A$.

An indecomposable $A$-$B$-bimodule $M$ is called {\em $\Bbbk$-split} if 
$M\simeq M_1\otimes_\Bbbk M_2$, for some indecomposable 
$M_1\in A\text{-mod}$ and $M_2\in \text{mod-}B$.
An $A$-$B$-bimodule is called {\em $\Bbbk$-split} if it is isomorphic
to a direct sum of $\Bbbk$-split indecomposable bimodules.

Note that any projective $A$-$B$-bimodule is $\Bbbk$-split. 
Indeed, if $P_1,\ldots ,P_n$ is are the indecomposable projective 
left $A$-modules, and $Q_1,\ldots ,Q_m$  are the indecomposable projective 
right $B$-modules (up to isomorphism), then
\begin{displaymath}
P_i\otimes_\Bbbk Q_j,\, i=1,\ldots ,n,\, j=1,\ldots ,m 
\end{displaymath}
are all indecomposable projective $A$-$B$-bimodules, up to isomorphism.

For any algebras $A$, $B$ and $C$,  whenever one of the bimodules ${}_AM_B$ or ${}_BN_C$ 
is $\Bbbk$-split,  then so is their tensor product ${}_AM\otimes_BN_C$, see \cite{MMZ}.
In particular, $\Bbbk$-split bimodules generate the minimal two-sided ideal in $\mathscr{BM}_\Bbbk$
with respect to inclusions.

\subsection{Two-sided preorder and equivalence}

Let $\mathscr{B}$ be a $\Bbbk$-linear, idempotent split and Krull-Schmidt bicategory
in which the horizontal composition $\circ$ is $\Bbbk$-bi\-li\-near.
As in \cite{MM1}, define the \textit{two-sided preorder} $\geq_J$ on the set of
isomorphism classes of indecomposable 1-morphisms of $\mathscr{B}$ by 
saying that $F\geq_J G$ if there are 1-morphisms $H_1,H_2$ such that $F$ is
isomorphic to a direct summand of $H_2\circ G\circ H_1$.
The induced equivalence relation is called \emph{two-sided equivalence} or \emph{$J$-equivalence},
and the equivalence classes are called \emph{two-sided cells}.

\begin{definition}
Let $A$ and $B$ be connected $\Bbbk$-algebras.
We say that $A\geq_J B$ if 
${}_AA_A\geq_J {}_BB_B$ in the bicategory $\mathscr{BM}_\Bbbk$.
Explicitly, $A\geq_J B$ if there are bimodules ${}_AM_B$ and ${}_BN_A$ such that
\begin{displaymath}
M\otimes_B N\simeq A\oplus X,
\end{displaymath}
for some $A$-$A$-bimodule $X$.
If $A\geq_J B$ and $B\geq_J A$, then we say that $A$ and $B$ 
are \emph{$J$-equivalent} or \emph{two-sided equivalent} and write $A\sim_JB$.
\end{definition}

The connectedness condition could be removed if we chose to define $A\geq_J B$ using the existence of ${}_AM_B$ and ${}_BN_A$ such that $A$ is a direct summand of $M\otimes_B N$ as $A$-$A$-bimodules. 
However, due to the additive nature of the definition, this provides no additional information.

If $A\sim_J B$, where $A$ is a direct summand of $M_1\otimes_B N_1$, and $B$ a direct summand of $N_2\otimes_A M_2$. Then we can set $M=M_1\oplus M_2$ and $N=N_1\oplus N_2$ to obtain that $A$ and $B$ are direct summands of $M\otimes_B N$ and $N\otimes_A M$, respectively. In this situation, we say that the two-sided equivalence is induced by $M$ and $N$.

We note that among all $\Bbbk$-algebras, the field $\Bbbk$ itself is maximal with respect to the two-sided order.

Moreover, when we restrict the two-sided preorder on $\mathscr{BM}$ to the monoidal category 
$$\mathscr{BM}_\Bbbk(A\text{-mod},A\text{-mod})=A\text{-mod-}A,$$ for some $A$, we have that the 
regular $A$-$A$-bimodule is minimal with respect to $\leq_J$.  Moreover, indecomposable 
$\Bbbk$-split $A$-$A$-bimodules form the maximal two-sided cell.

\subsection{Equivalences of Morita type}
Recall that two algebras $A$ and $B$ are {\em Morita equivalent} 
if the categories $A$-mod and $B$-mod are equivalent. An equivalent 
condition is the existence of bimodules ${}_AP_B$ and ${}_BQ_A$ such that we have the following isomorphisms of bimodules.
\begin{displaymath}
P\otimes_B Q \simeq A\quad\text{  and }\quad Q\otimes_A P\simeq B.
\end{displaymath}
With this formulation, we see that Morita equivalence is a special case of two-sided equivalence.

Given an abelian category $\mathcal{A}$, the {\em stable category $\underline{\mathcal{A}}$} 
is defined as the quotient of $\mathcal{A}$ by the ideal generated by projective objects.
This means that it has the same objects as $\mathcal{A}$ 
and $\underline{\mathcal{A}}(X,Y):=\mathcal{A}(X,Y)/\sim$, where $f\sim g$ if $f-g$ factors through a projective object of $\mathcal{A}$.

Also, recall that the {\em singularity category $D_\mathrm{sg}(A)$} is defined as the Verdier quotient
\begin{displaymath}
D_\mathrm{sg}(A) := D^b(A\text{-mod})/K^b(A\text{-proj}),
\end{displaymath}
where $D^b(A\text{-mod})$ is the bounded derived category $A$-mod, and $K^b(A\text{-proj})$ is the bounded homotopy category of finite-dimensional projective $A$-modules.

The algebras $A$ and $B$ are called {\em stably equivalent} 
if the stable categories $A$-$\underline{\text{mod}}$ and $B$-$\underline{\text{mod}}$ are equivalent. Following \cite{Br}, we say that a stable equivalence is of {\em Morita type} if it is induced by tensoring with  left-rigth projective bimodules ${}_AM_B$ and ${}_BN_A$ such that
\begin{equation}\label{eqnn1}
M\otimes_B N\simeq A\oplus P\quad \text{ and }\quad N\otimes_A M\simeq B\oplus Q,
\end{equation}
where ${}_AP_A$ and ${}_BQ_B$ are projective bimodules.

Similarly, $A$ and $B$ are {\em singularly equivalent} if 
their singularity categories $D_{sg}(A)$ and $D_{sg}(B)$ are 
equivalent, and singularly equivalent of  {\em Morita type} if 
the equivalence is induced by   left-rigth projective bimodules 
${}_AM_B$ and ${}_BN_A$ satisfying  \eqref{eqnn1},
where ${}_AP_A$ and ${}_BQ_B$ have finite projective 
dimensions as bimodules, see \cite{CS} and \cite{ZZ}.

Finally, if we remove all conditions on $P$ and $Q$ above and obtain a two-sided equivalence induced by left-right projective bimodule $M$ and $N$, then $A$ and $B$ are called {\em separably equivalent}. 
This notion was introduced in \cite{Ka1} and \cite{Li}.
We adopt the terminology from \cite{Pe} and say that 
{\em $A$ separably divides $B$} if there are left-right projective bimodules ${}_AM_B$ and ${}_BN_A$ such that the regular $A$-$A$-bimodule is a direct summand of $M\otimes_B N$.
It follows from the definitions that if $A$ separably divides $B$, then $B\leq_J A$.

As in \cite{Pe}, we say that two separably equivalent algebras $A$ and $B$ are 
{\em symmetrically separably equivalent} if the functors
\begin{displaymath}
M\otimes_B- :B\text{-mod}\to A\text{-mod} \quad
\text{ and }\quad N\otimes_A -:A\text{-mod}\to B\text{-mod} 
\end{displaymath}
are biadjoint. 
Symmetric separable equivalence is further characterized in \cite{Ka2}.
We remark that if $A$ and $B$ are symmetric, then they are separably equivalent if and only if they are symmetrically separably equivalent.

Summarizing, assume that there are left-right projective bimodules ${}_AM_B$ and ${}_BN_A$ such that $M\otimes_BN\simeq A\oplus X$ and $N\otimes_A M\simeq B\oplus Y$ as bimodules. We have the following summary table.

\begin{center}
\begin{tabular}{l|l}
$X$, $Y$ & Equivalence \\
\hline
Zero & Morita equivalence\\
\hline
Projective & Stable equivalence of Morita type\\
\hline
Finite projective dimension & Singular equivalence of Morita type\\
\hline
No condition & Separable equivalence
\end{tabular}
\end{center}

The following chain of implications holds.
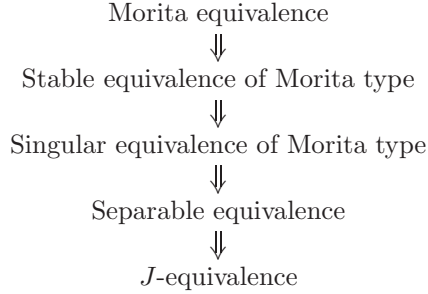
\begin{displaymath}
\xymatrix@R=10pt{
\text{Morita equivalence} \ar@{=>}[d]\\
\text{Stable equivalence of Morita type} \ar@{=>}[d]\\
\text{Singular equivalence of Morita type} \ar@{=>}[d]\\
\text{Separable equivalence} \ar@{=>}[d]\\
J\text{-equivalence}
}
\end{displaymath}
The notions of stable and separable equivalence of Morita type were introduced in the study of group algebras.
As far as we are aware, the literature does not provide many explicit examples of equivalent algebras.

Seaprable division turns out to be more useful in us. Section~\ref{s_reptype} of this paper is devoted to generalizing results of \cite{Pe} from separable to two-sided equivalence. Moreover, we mention some further facts.
\begin{itemize}
\item If $A$ and $B$ are self-injective algebras which are derived equivalent, then they are stably equivalent of Morita type, see \cite{Ri}.
\item If $G$ is a finite group, $\Bbbk$ a field of positive characteristic $p$, and $P$ a Sylow $p$-subgroup of $G$, then $\Bbbk G$ and $\Bbbk P$ are separably equivalent, see \cite{Li}.
\item If $A$ and $B$ are Nakayama algebras which are quotients of hereditary algebras, then $A$ and $B$ are stably equivalent of Morita type if and only if they are Morita equivalent, see \cite{Po1}.
\end{itemize}

\section{Structure of bimodules inducing equivalence}

\subsection{Generators}
In this subsection we fix algebras $A$ and $B$, and bimodules ${}_AM_B$, ${}_BN_A$, such that
\begin{align*}
{}_AM\otimes_BN_A\simeq A\oplus X,
\end{align*}
for some $A$-$A$-bimodule $X$.

The following result generalizes  the fact that the left-right projective bimodules inducing a separable division are one-sided projective generators, see e.g. \cite[Lemma~2.2]{Ka2}.

\begin{lemma}\label{lemma_top}
If $P(M)$ and $P(N)$ are projective covers of $M$ and $N$ in 
$A\text{\rm-mod-}B$ and $B\text{\rm-mod-}A$, respectively, 
then ${}_AP(M)$ and $P(N)_A$ are projective generators 
in $A\text{\rm-mod}$ and $\text{\rm mod-}A$, respectively.
\end{lemma}

\begin{proof}
Let $P_1,\ldots ,P_n$ and $P_1',\ldots ,P_n'$
be complete sets of pairwise non-isomorphic projective left
and right $A$-modules, respectively.
By \cite{Ha}, the bimodule 
${}_AA_A$ has minimal projective cover $\bigoplus_{i=1}^nP_i\otimes P_i'$.
As ${}_AA_A$ is a direct summand of $M\otimes_B N$, the projective cover of ${}_AA_A$ 
is a direct summand of the projective cover $P(M\otimes_BN)$ of $M\otimes_B N$ in $A\text{-mod-}A$.

Consider now some minimal projective covers  $P(M)\xrightarrow{p}M$  and $P(N)\xrightarrow{\pi}N$.
We have that
\begin{displaymath}
\xymatrix{
P(M)\otimes_B P(N) \ar[rr]^{p\otimes \pi}&& M\otimes_B N
}
\end{displaymath}
is an epimorphism from a projective module to $M\otimes_B N$. Hence it factors through 
any minimal projective cover of $M\otimes_B N$ via a split epimorphism. This means that
\begin{align*}
\bigoplus_{i=1}^nP_i\otimes P_i'
\end{align*}
is a direct summand of $P(M)\otimes_B P(N)$.

Suppose that $Q_1,\ldots ,Q_m$ and $Q_1',\ldots ,Q_m'$ are all pairwise 
non-isomorphic indecomposable projective left and right $B$-modules, respectively. 
Then any indecomposable direct summand of $P(M)$ is of the form
$P_i\otimes_\Bbbk Q_j'$, for some $i,j$, and any indecomposable direct 
summand of $P(N)$ is of  the form $Q_k\otimes_\Bbbk P_l'$, for some $k,l$.
The tensor product over $B$ will result in summands of the form
\begin{displaymath}
P_i\otimes_\Bbbk Q_j'\otimes_B Q_k\otimes_\Bbbk P_l' \simeq 
(P_i\otimes_\Bbbk  P_l')^{\oplus \dim (Q_j'\otimes_B Q_k)}.
\end{displaymath}
Therefore, to have a summand of the form $\bigoplus_{i=1}^nP_i\otimes P_i'$, 
each $P_i$ must occur as a summand in ${}_AP(M)$, and each $P_i'$ must occur
as a summand in $P(N)_A$. This proves our lemma.
\end{proof}

Lemma~\ref{lemma_top} implies, for example, that all simple 
left $A$-modules are summands of ${}_A\mathrm{top}({}_AM_B)$. 
% We remark also that it is stronger than the claim that the projective cover  of ${}_AM$ in $A$-mod is a projective generator.
Moreover, the following lemma implies that, 
if $A$ is self-injective, then ${}_AM$ and $N_A$ themselves are generators.

\begin{lemma}
\begin{enumerate}[(i)]
\item The left module ${}_AM$ is faithful, and all indecomposable projective-injective left $A$-modules are direct summands of ${}_AM$.
\item The right module $N_A$ is faithful, and all indecomposable projective-injective right $A$-modules are direct summands of $N_A$.
\end{enumerate}
\end{lemma}

\begin{proof}
We prove $(i)$. First of all, ${}_AM$ is faithful. Indeed, if $a\in \mathrm{Ann}({}_AM)$, then
\begin{displaymath}
{}_A0_A={}_A0_B\otimes_B N = aM\otimes_B N = aA\oplus aX,
\end{displaymath}
so that $a\in \mathrm{Ann}({}_AA)=\{0\}$.

The representation map $A\mapsto \mathrm{End}_\Bbbk(M)$ is a homomorphism of left $A$-modules.
Hence the fact that $M$ is faithful implies that the left regular module 
${}_AA$ embeds into $\mathrm{End}_\Bbbk({}_AM)\cong {}_AM^{\oplus \dim(M)}$.
Hence, if $P$ is a projective left $A$-module, then, for some $n$, there is a monomorphism
\begin{displaymath}
P\hookrightarrow {}_AM^{\oplus n}.
\end{displaymath} 
If $P$ is also injective, then  this monomorphism splits, so that $P$ is a direct summand of ${}_AM^{\oplus n}$. If, additionally, $P$ is indecomposable, then $P$ must be a summand of ${}_AM$.
\end{proof}

\subsection{Adjoint pairs}
Assume that $B\leq_J A$, with bimodules ${}_AM_B$ and ${}_BN_A$ such that ${}_AA_A$ is a direct summand of $M\otimes_B N$.
A natural question is when the functors  $M\otimes_B- :B\text{-mod}\to A\text{-mod}$ and $N\otimes_A- :A\text{-mod}\to B\text{-mod}$ form an adjoint pair $(M,N)$.
From \cite[Lemma~2.5]{Ka2} it follows that the counit of such adjunction is a split epimorphism.
We also have the following general result.

\begin{lemma}\label{lemma_adj}
For $M\in A\text{-mod-}B$ and $N\in B\text{-mod-}A$, the functors $(M\otimes_B-,N\otimes_A-)$ is an adjoint pair of functors if and only if $M$ is projective as a left $A$-module and $\mathrm{Hom}_{A\text{-}} (M,A)\simeq N$ as $B$-$A$-bimodule.
\end{lemma}
\begin{proof}
Mutatis mutandis the proof of \cite[Lemma~13]{MZ}.
\end{proof}

In the situation of the above lemma, the functor $N\otimes_A -\simeq \mathrm{Hom}_{A-}(M,-)$ is exact. Consequently, $N$ is projective as right $A$-module. 
In light of Lemma~\ref{lemma_top}, the following is clear.

\begin{corollary}
If ${}_AM_B$, ${}_BN_A$ are bimodules such that ${}_AA_A$ is a direct summand of $M\otimes_BN$
and $(M\otimes_B -,N\otimes_A-)$ is an adjoint pair of functors, then ${}_AM$ and $N_A$ are generators.
\end{corollary}

If the functors $M\otimes_B-$ and $N\otimes_A -$ are biadjoint, the argument above also yields the following result.

\begin{corollary}\label{cor_adj}
If ${}_AM_B$ and ${}_BN_A$ induce a two-sided equivalence between $A$ and $B$, and the functors $M\otimes_B- :B\text{-mod}\to A\text{-mod}$ and $N\otimes_A- :A\text{-mod}\to B\text{-mod}$ are biadjoint, then $A$ and $B$ are symmetrically separably equivalent.
\end{corollary}

If two algebras are stably equivalent of Morita type, then there always exist bimodules forming a  biajoint pair which generate this equivalence, see \cite{DM}. Moreover, bimodules inducing symmetrically separable equivalence also induce biadjoint functors between the stable module categories, and these functors give a (generalized) symmetric separable equivalence between the stable  module categories (see \cite[Section~4]{Pe}, and Section~5.1 in this paper).
In this sense, symmetric separable equivalence is close to stable equivalence of Morita type.
It was shown in \cite[Corollary~2.4]{Liu} that stable equivalence of Morita type preserves the property 
for an algebra of being symmetric.
It is therefore natural to ask whether the same is true for symmetric separable equivalence.
In Section~\ref{s_examples} we shall give a negative answer to this question.
However, the following lemma implies that any algebra which is symmetrically separably equivalent to a symmetric algebra is self-injective.

\begin{lemma}
Let $A$ and $B$ be algebras for which there exist bimodules ${}_AM_B$  and ${}_BN_A$
such that $N\otimes_A M\simeq B\oplus Q,$ for some left-right projective bimodule $Q$.
Assume further that these bimodules give rise to a biadjoint pair of functors between 
the categories $A\text{-mod}$ and $B\text{-mod}$. If $A$ is symmetric, then $B$ is self-injective.
\end{lemma}

\begin{proof}
Let $D(-)=\mathrm{Hom}_\Bbbk (-,\Bbbk)$ be the usual $\Bbbk$-duality.
Recall that the assumption that $A$ is symmetric means that there is an isomorphism
${}_A A_A\cong {}_AD(A)_A$ of $A$-$A$-bimodules.
This also implies that the functors $\mathrm{Hom}_\Bbbk (-,\Bbbk)$
and $\mathrm{Hom}_A (-,A)$ between $A$-mod and mod-$A$, in each direction, are isomorphic.
Indeed, using the usual tensor-hom adjunction, we have:
\begin{equation}\label{eqnn2}
\mathrm{Hom}_A (-,A)\cong
\mathrm{Hom}_A (-,\mathrm{Hom}_\Bbbk (A,\Bbbk))\cong
\mathrm{Hom}_\Bbbk (A\otimes_A -,\Bbbk)\cong
\mathrm{Hom}_\Bbbk (-,\Bbbk).
\end{equation}

We shall now follow the computation from the proof of 
\cite[Corollary~2.4]{Liu}, using several of the classical tricks for 
homs and tensor products involving projective modules.
To start with, using \eqref{eqnn2} and the fact that ${}_AM$ is projective, 
we obtain, by \cite[Lemma~3.55]{Ro}, that
\begin{align*}
N \otimes_A M &\simeq \mathrm{Hom}_{-A}\big( \mathrm{Hom}_{-A} (N,A),A\big) \otimes_A M \\
&\simeq  \mathrm{Hom}_{-A} (\mathrm{Hom}_{A-}(M,\mathrm{Hom}_{-A} (N,A)),A) \\
&\simeq  D\big( \mathrm{Hom}_{A-}(M,\mathrm{Hom}_{-A} (N,A))\big) .
\end{align*}
% \begin{align*}
% N \otimes_A M &\simeq D\big( \mathrm{Hom}_{-A} (N,A)\big) \otimes_A M \\
% &\simeq  D\big( \mathrm{Hom}_{A-}(M,\mathrm{Hom}_{-A} (N,A))\big) \\
% &\simeq  D\big( \mathrm{Hom}_{A-}(M,A\otimes_A\mathrm{Hom}_{-A} (N,A))\big) .
% \end{align*}
Substituting $\mathrm{Hom}_{-A} (N,A)\cong A\otimes_A\mathrm{Hom}_{-A} (N,A)$, 
the projectivity of ${}_AM$ allows us to use \cite[Proposition~20.10]{AF} to obtain
\begin{align*}
D\big( \mathrm{Hom}_{A-}(M,A\otimes_A\mathrm{Hom}_{-A} (N,A))\big) &\simeq D\big( \mathrm{Hom}_{A-}(M,A) \otimes_A\mathrm{Hom}_{-A} (N,A)\big) \\
&\simeq D( N\otimes_A M) \\
&\simeq D(B\oplus Q),
\end{align*}
where the isomorphism between the first and second rows follows from the assumption about biadjointness.
This gives 
\begin{align*}
B\oplus Q\simeq N\otimes_A  M \simeq D(B)\oplus D(Q).
\end{align*}
Since $B$ and $Q$ are left-right projective $B$-$B$-bimodules, $D(B)$ and $D(Q)$ are left-right injective. 
Hence $B$ is injective, both as a left and as a right $B$-module.
\end{proof}

\section{Quotients and subalgebras}
In this section we shall consider the $J$-relations of an algebra with its factor algebras,
as well as $J$-relations between an algebra $A$ and subalgebra $B\subseteq A$ induced by the natural bimodule structures on $A$.

\subsection{Quotients}
Let $\varphi :B\to A$ be a homomorphism of unital $\Bbbk$-algebras, i.e. such that $\varphi (1_B)=1_A$.
Then $A$ has the natural structure of 
a $B$-$B$-bimodule by defining the left (resp. right) action of $b\in B$ on $A$ via 
the left (resp. right) multiplication by $\varphi (b)$. We just 
say that $B$ acts on $A$ via $\varphi$. We denote the action of $b\in B$ on $a\in A$ by $\bullet$, 
for example, the left $B$-action on $A$ via $\varphi$ is written
\begin{displaymath}
b\bullet a := \varphi (b)a.
\end{displaymath}

The following result is fairly expected.

\begin{proposition}\label{prop_surj}
If there is a surjective algebra homomorphism $\varphi :B\to A$, then $B\leq_J A$.
\end{proposition}

\begin{proof}
We shall prove that the $A$-$A$-bimodule $A\otimes_B A$ has a direct summand isomorphic to $A$. Consider first the bimodule morphism $\psi: A\otimes_B A\to A$ defined on simple tensors by $a\otimes a'\mapsto aa'$.
First of all, $\psi$ is surjective, since $A\otimes_B A$ contains the element $a\otimes 1_A$ for all $a\in A$, and $\psi (a\otimes 1)=a$.
Moreover, $\psi$ is right split, with retract $\eta$,
where $\eta$ is the linear map
\begin{displaymath}
\eta: A\to A\otimes_B A,\ a\mapsto a\otimes 1.
\end{displaymath}
We need to show that $\eta$ is an $A$-$A$-bimodule morphism. Once this is verified, it is obvious that $\eta \circ \psi =\mathrm{id}_A$.
That $\eta$ is a morphism of left $A$-modules is clear. Since $\varphi$ is surjective, $\eta$ is also a right $A$-module morphism. Indeed, given $a,a'\in A$, fix some $b\in B$ such that $\varphi (b)=a'$. Then
\begin{align*}
\eta (aa') &= (aa')\otimes_B 1_A = \big( a\varphi (b)\big) \otimes_B 1_A = \\
&= a\bullet b\otimes _B 1_A =\\
&= a\otimes_B b\bullet 1_A = \\
&= a\otimes_B \varphi(b)1_A =\\
&= a\otimes_B a' =\\
&= (a\otimes 1_A)a' = \eta (a)a'.
\end{align*}
This concludes the proof.
\end{proof}

We remark that this construction does not induce $J$-equivalence unless the surjective morphism $\varphi :B\to A$ is an isomorphism.
Indeed, if $B$ is a direct summand of ${}_BA_B$, then $A$ must be faithful as $B$-module, implying that $\ker \varphi$ is trivial.
However, this does not in itself exclude the possibility of $A$ and $B$ being $J$-equivalent.

\subsection{Subalgebras}
Let $B\subseteq A$ be a subalgebra.
Recall that $A\supseteq B$ is a 
\begin{itemize}
\item \emph{split extension} if $B$ is a direct summand of $A$ as $B$-$B$-bimodules;
\item \emph{separable extension} if $A$ is a direct summand of $A\otimes_B A$ as $A$-$A$-bimodules;
\end{itemize}
see further e.g. \cite{Ka1}.
Clearly, if $A\supseteq B$ is split, then $B\geq_J A$, and if $A\supseteq  B$ is separable, then $A\geq_J B$.

Split extensions $A\supseteq B$ can also be characterised by a surjective $B$-$B$-bimodule morphism $E:A\to B$ such that $E|_B=\mathrm{id}_B$.

One useful criteria for determining whether an extension $A\supseteq B$ is separable is given in \cite[Proposition~1.2]{Ka3}: if $\Bbbk$ is perfect, and $B\subseteq A$ is a subalgebra such that $\mathrm{rad}(B)$ is an ideal in $A$, then $A\supseteq B$ is separable if and only if $\mathrm{rad}(B)=\mathrm{rad}(A)$.

\begin{lemma}\label{lemma_sep}
Assume that $\Bbbk$ is a perfect field and  $A=\Bbbk Q/I$ the path algebra of a quiver modulo an admissible ideal.
Then any partition $Q_0=X_1\cup \ldots \cup X_m$, where the $X_i$'s are pairwise disjoint nonempty subsets of $Q_0$,  gives rise to a subalgebra $B$ such that $A\supseteq B$ is a separable extension.
\end{lemma}

\begin{proof}
We construct a quiver $Q'$ and an admissible ideal $I'$ such that $B=\Bbbk Q'/I'$ by ''contracting'' in $Q$ according to the partition $Q_0=\cup_{i=1}^m X_i$.
Set $Q'_0=\{ 1,\ldots ,m\}$.
The arrows in $Q'$ are in bijection with the arrows in $Q$: for each arrow $i\xrightarrow{\alpha}j$ in $Q$ there is an arrow $s_i\xrightarrow{\alpha'}s_j$ in $Q'_0$,
where $i\in X_{s_i}$ and $j\in X_{s_j}$.
The relations are also inherited from $A$: if $a\in I$, then the corresponding path $a'$ in $\Bbbk Q'$ is in $I'$.
Then there is an injective algebra homomorphism $\varphi :B\to A$ determined by
\begin{align*}
\varepsilon_i'&\mapsto \sum_{x\in X_i} \varepsilon_x,\quad i=1,\ldots ,m\\
\alpha '&\mapsto \alpha.
\end{align*}
As subalgebra $B\subseteq A$ we have $\mathrm{rad}(A)=\mathrm{rad}(B)$, so by \cite[Prop~1.2]{Ka3}, $A\supseteq B$ is a separable extension.
\end{proof}

\begin{example}\label{ex_sep}
Assume that $\Bbbk$ is perfect.
\begin{enumerate}[(i)]
\item Let $A=\Bbbk Q/I$, where $Q$ is
\begin{displaymath}
\xymatrix{1\ar[r]^\alpha & 2 \ar[r]^\beta & 3\ar[r]^\gamma & 4}
\end{displaymath}
and $I=(\beta \alpha )$.
Consider the partition
\begin{align*}
Q_0=\{1\} \cup \{ 2\} \cup \{3,4\} .
\end{align*}
Then $B$ is the subalgebra isomorphic to $\Bbbk Q'/I'$, where
\begin{displaymath}
Q'=\xymatrix{
1\ar[r]^\alpha & 2\ar[r]^\beta & 3\ar@(ul,ur)[]^\gamma
}
\end{displaymath}
and $I'=( \beta \alpha ,\gamma^2)$. Here, the node 3 is the ''contraction'' of 3 and 4 in $Q$.
By the above lemma $A\supseteq B$ is separable, and hence $A\geq_J B$.
\item Denote by $\Theta$ the path algebra of the Kronecker quiver
\begin{displaymath}
\xymatrix{
1\ar@/^/[r]^\alpha \ar@/_/[r]_\beta & 2,
}
\end{displaymath}
and by $A_3'$ the path algebra of the non-uniformly oriented quiver of Dynkin type $A$ on 3 vertices
\begin{displaymath}
\xymatrix{
1\ar[r]^\alpha & 2 & 3\ar[l]_{\beta}.
}
\end{displaymath}
Then the partition $\{ 1,3 \}\cup \{2\}$ of the nodes of the latter yields, as in the proof of  Lemma~\ref{lemma_sep}, a subalgebra of $A_3'$ isomorphic to $\Theta$, such that $A_3'\supseteq \Theta$ is separable.
\end{enumerate}
\end{example}

\subsection{Invariants under group action}
Let $A$ be a $\Bbbk$-algebra, and $G$ a finite abelian group acting on $A$ via automorphisms. Denote by $A^G$ the subalgebra of $G$-invariant elements of $A$, i.e.
\begin{align*}
A^G=\{ a\in A\mid g(a)=a \ \forall g\in G\} .
\end{align*}

For an element $g\in G$, we denote by ${}^gA$ and $A^g$ the 
$A$-$A$-bimodule $A$, where the left or the right $A$-action is twisted by $g$, respectively. 

\begin{theorem}\label{thm_group}
\begin{enumerate}[(i)]
\item If $\mathrm{char}(\Bbbk)$ does not divide the order of $G$, then $A\leq_J A^G$.
\item If $A$ is basic and the action of $G$ is such that the induced action on the underlying 
quiver of $A$ is free, then $A^G\leq_J A$.
\item Assume that $A$ is basic, $\mathrm{char}\, (\Bbbk )\nmid G$, the action of $G$ on the underlying quiver of $A$ is free, $A^G$ is symmetric, and $\mathrm{Hom}_\Bbbk(A,\Bbbk) \simeq {}_AA^{g}_A$ for some $g\in G$. Then $A$ and $A^G$ are symmetrically separably equivalent. 
\end{enumerate}
\end{theorem}

\begin{proof}
\begin{enumerate}[(i)]
\item 
To prove the first claim, we shall show that $A^G$ is a direct summand of 
$A\cong A\otimes_A A$ as an $A^G$-$A^G$-bimodule.
Indeed, as $A$ is a $G$-module and $\mathrm{char}(\Bbbk) \nmid |G|$, we can write $A$ as a direct sum of isotypic components with respect to the $G$-action.
Since $G$ is abelian, we have that
\begin{align*}
A = \bigoplus_{\xi \in \Xi} A_{\xi},
\end{align*}
where $\Xi$ is the (finite) set of all $G$-characters and
\begin{align*}
A_\xi = \{ a\in A\mid ga = \xi (g) a\ \forall g\in G\} .
\end{align*}
Then $A^G$ is exactly the trivial component corresponding to the trivial character $\mathrm{triv}:G\to \Bbbk$, which sends $g\mapsto 1$ for all $g\in G$.
Moreover, each $A_\xi$ is an $A^G$-$A^G$-bimodule. Indeed, if $a\in A^G= A_\mathrm{triv}$ and $b\in A_\xi$, then
\begin{align*}
g (ab) = g(a)g(b) = a \xi (g) b = \xi (g) ab
\end{align*}
so $ab\in A_\xi$. Similarly $ba\in A_\xi$. Hence $A = \bigoplus_{\xi \in \Xi} A_{\xi}$ is a decomposition of $A$ into a direct sum of $A^G$-$A^G$-bimodules, which has a summand isomorphic to $A^G$.
This implies that $A\leq_J A^G$.

\item Consider the $A$-$A$-bimodule $A\otimes_{A^G} A$. 
We claim that, under our assumptions, it is isomorphic to
\begin{align*}
Q=\bigoplus_{g\in G} {}^gA.
\end{align*}
Given this, if we denote the unit of $G$ by $e$, we have ${}^eA=A$
which implies the statement we are trying to prove.

Denote by $1_g$ the element $1$ in ${}^gA$.
Consider the linear map
\begin{align*}
\varphi :A\otimes_{A^G} A &\to Q \\
a\otimes b &\mapsto \sum_{g\in G} g(a)1_gb .
\end{align*}
First of all, we note that this morphism is well-defined, since, for any $a\in A^G$, and any $b,c\in A$,
\begin{align*}
\varphi (ba\otimes c)  &= \sum_{g\in G} g(ba) 1_g c =\\
&= \sum_{g\in G} g(b)g(a) 1_g c =\\
&= \sum_{g\in G} g(b) a 1_g c =\\
&= \sum_{g\in G} g(b) 1_g a c = \varphi (b\otimes ac).
\end{align*}
Moreover, $\varphi$ is a homomorphism of $A$-$A$-bimodules. 
Indeed, denote both the left and the right action by $A$ on ${}^gA$ by $\bullet$. Then
\begin{align*}
\varphi \big( (ab)\otimes (cd) \big) &= \sum_{g\in G} g(ab)1_g (cd) =\\
&=\sum_{g\in G} g(a)g(b)1_g cd =\\
&=  \sum_{g\in G}a \bullet g(b)1_g c \bullet d=\\
&= a \bullet \left( \sum_{g\in G} g(b)1_g c\right) \bullet d =  a \bullet \varphi (b\otimes c) \bullet d .
\end{align*}
It is clear that $\varphi$ is surjective. Let us now compare dimensions. First, we have $\dim Q= |G|\cdot \dim A$.

Next we note that the rank of the free $A^G$-module $A$ is $|G|$. Indeed, recall 
our assumption that the action on the underlying quiver is free. By Burnside's lemma, it therefore follows that the number of orbits of the $G$-action is $\frac{\dim A}{|G|}$. This is the dimension of $A^G$. Since $A$ is free over $A^G$, the rank must be $|G|$.

This, in turn, implies that $\dim (A\otimes_{A^G} A)=|G|\dim(A)=\dim Q$. Consequently, $\varphi$ is an isomorphism, so the proof is complete.

\item By $(i)$ and $(ii)$, we already know that $A$ is a direct summand of ${}_AA\otimes_{A^G}A_A$, and $A^G$ is a direct summand of ${}_{A^G}A\otimes_AA_{A^G}$.
Moreover, $A$ is free as (left and right) $A^G$-module, so both ${}_AA_{A^G}$ and ${}_{A^G}A_A$ are left-right projective. Hence $A$ and $A^G$ are separably equivalent.
It remains to show that the functors
\begin{align*}
{}_AA\otimes_{A^G} - &:A^G\text{-mod}\to A\text{-mod}\\ 
{}_{A^G}A\otimes_{A} - &:A\text{-mod}\to A^G\text{-mod}
\end{align*}
are biadjoint.
By Lemma~\ref{lemma_adj}, it is enough to prove that
\begin{align}
\mathrm{Hom}_{A-} ({}_AA_{A^G},{}_AA_{A}) &\simeq {}_{A^G}A_{A} \label{eq:adj1} \\
\mathrm{Hom}_{A^G-} ({}_{A^G}A_{A},{}_{A^G}A^G_{A^G}) &\simeq {}_{A}A_{A^G}. \label{eq:adj2}
\end{align}
The first statement is immediate. To prove \eqref{eq:adj2}, we use that $A^G$ is assumed to be symmetric and apply the tensor-hom adjunction:
\begin{align*}
\mathrm{Hom}_{A^G-} ({}_{A^G}A_{A},{}_{A^G}A^G_{A^G}) &\simeq \mathrm{Hom}_{A^G-} ({}_{A^G}A_{A},{}_{A^G}\mathrm{Hom_\Bbbk}(A^G,\Bbbk )_{A^G}) \simeq \\
&\simeq \mathrm{Hom}_\Bbbk ({}_{A^G}A_A,\Bbbk )\simeq \\
&\simeq \mathrm{Hom}_\Bbbk ({}_{A^G}A_A,\Bbbk ) \otimes_{A^G} A^G \simeq\\
&\simeq \mathrm{Hom}_\Bbbk ({}_{A}A_A,\Bbbk ) \otimes_A A\otimes_{A^G} A^G .
\end{align*}
By assumption, $\mathrm{Hom}_\Bbbk ({}_{A}A_A,\Bbbk )\simeq {}_AA_A^g$ for some $g\in G$, so
\begin{displaymath}
 \mathrm{Hom}_\Bbbk ({}_{A}A_A,\Bbbk ) \otimes_{A^G} A^G \simeq {}_AA^g\otimes_{A^G} A^G .
\end{displaymath}
However, since for any $a\in A^G$ we have by definition $g(a)=a$, it follows that
\begin{align*}
{}_AA_A^g\otimes_{A^G} A^G \simeq {}_AA\otimes_{A^G} A^G \simeq {}_AA_{A^G}.
\end{align*}
This concludes the proof.
\end{enumerate}
\end{proof}

We remark that if the conditions of $(i)$ and $(ii)$ both hold, then $A$ and $A^G$ are separably equivalent.
The additional conditions in part $(iii)$ 

\begin{example}
Let $\Bbbk$ be a field such that $\mathrm{char}\, (\Bbbk )\neq 2$, and let $Q$ be the quiver
\begin{displaymath}
\xymatrix{
1\ar@/^/[r]^\alpha & 2\ar@/^/[l]^\beta
} .
\end{displaymath}
Set $A=\Bbbk Q/(\alpha \beta,\, \beta \alpha )$.
Consider the cyclic group $C_2=\left\langle c\right\rangle$ and let $c$ act on $Q$ by
\begin{align*}
1\leftrightarrow 2,\quad \alpha \leftrightarrow \beta .
\end{align*}
This free action on $Q$ extends to an automorphism of $A$.
The algebra of ivnariants under this action of $C_2$ is spanned by $1$ and $\alpha +\beta$, and there is an obvious isomorphism with the dual numbers $D=\Bbbk [x]/(x^2)$ - a symmetric algebra.

A basis of $A$ is $\{ \varepsilon_1,\varepsilon_2,\alpha ,\beta\}$, where $\varepsilon_i$ is the path of length zero  at $i$.
Denote by $f_1,f_2,f_\alpha ,f_\beta$ the corresponding dual basis elements.
It is straightforward to check that the nonzero left and right actions of $A$ on the basis elements are as follows.
\begin{displaymath}
\begin{array}{c|c|c|c|c}
\bullet & f_1 & f_2 & f_\alpha & f_\beta \\
\hline 
\varepsilon_1 & f_1 & 0 & f_\alpha & 0\\
\hline 
\varepsilon_2 & 0 & f_2 & 0 & f_\beta \\
\hline 
\alpha & 0 & 0 & f_2 & 0\\
\hline 
\beta & 0 & 0 & 0 & f_1
\end{array} 
\hspace{1cm}
\begin{array}{c|c|c|c|c}
\bullet & \varepsilon_1 & \varepsilon_2 & \alpha & \beta \\
\hline
f_1 & f_1 & 0 & 0 & 0\\
\hline
f_2 & 0 & f_2 & 0 & 0\\
\hline
f_\alpha & 0 & f_\alpha & f_1 & 0\\
\hline
f_\beta & f_\beta & 0 & 0 & f_2
\end{array}
\end{displaymath}
Now
\begin{align*}
f_\alpha &\mapsto \varepsilon_1\\
f_\beta &\mapsto \varepsilon_2\\
f_2&\mapsto \alpha\\
f_1&\mapsto \beta
\end{align*}
extends  to an isomorphism of $A$-$A$-bimodules $\mathrm{Hom}_\Bbbk (A,\Bbbk )\to A^c$. By Theorem~\ref{thm_group}$(iii)$, $A$ and $D$ are symmetrically separably equivalent.
\end{example}

\begin{remark}
In the situation of the above theorem, we saw in the proof of part $(iii)$ that, if $\mathrm{char}\, \Bbbk \nmid |G|$ and the action of $G$ on the underlying quiver is free, then $A$ and $A^G$ are separably equivalent.
\end{remark}

\subsection{Skew group algebras}
In this section we shall recall the notion of skew group algebras, see further e.g. \cite{RR}.
This turns out to be rich source of examples of $J$-equivalent algebras.

Let $A$ be  a $\Bbbk$-algebra and $G$ a group acting on $A$ via homomorphisms.
Define the \emph{skew group algebra} $A\ast G$ as follows.
As vector space, $A\ast G$ is $A\otimes_\Bbbk \Bbbk G$, and the multiplication is given by
\begin{displaymath}
(a\otimes g)(a'\otimes g') = ag(a')\otimes gg'.
\end{displaymath}
There is a natural algebra monomorphism $\iota :A\to A\ast G$ given by $\iota (a)=A\otimes e$, where $e\in G$ is the identity element. Hence $A$ is a subalgebra of $A\ast G$, and $\iota$ induces an action of $A$ on $A\ast G$.

\begin{theorem}\label{thm_skew}
If $\mathrm{char}\, (\Bbbk )$ does not divide the order of $G$, then $A$ and $A\ast G$ are symmetrically separably equivalent.
\end{theorem}

\begin{proof}
By \cite[Theorem~1.1]{RR}, we have that $A$ and $A\ast G$ satisfy the following properties.
\begin{enumerate}[(A)]
\item 
\begin{enumerate}[(i)]
\item $A$ is a direct summand of $A\ast G$ as $A$-$A$-bimodule.
\item The multiplication map $A\ast G\otimes_A A\ast G\to A\ast G$ is a split epimorphism of $A\ast G$-$A\ast G$-bimodules.
\end{enumerate}
\item The functors ${}_{A\ast G}A\ast G\otimes_A -$ and ${}_AA\ast G\otimes_{A\ast G} -$ are biadjoint.
\item $\mathrm{rad}(A)(A\ast G)=(A\ast G) \mathrm{rad} (A) = \mathrm{rad} (A\ast G)$.
\end{enumerate}
By (A)(i), $A\geq_J A\ast G$, and by (A)(ii), $A\leq_J A\ast G$.
Moreover, (B) together with Corollary~\ref{cor_adj} implies that $A$ and $A\ast G$ are symmetrically separably equivalent.
\end{proof}

We remark also that  \cite[Theorem~1.3]{RR} gives many nice properties of the relationship between  $A$ and $A\ast G$.
For  example, $A$ and $A\ast G$ have the same Loewy length, and $\mathrm{rad}(A\ast G)=\mathrm{rad}(A)\otimes_\Bbbk \Bbbk G$.

In Section~\ref{s_examples}, we will provide many examples arising from the above theorem. More can be found in e.g. \cite{RR}, \cite{De} and \cite{HY}.

\section{Representation type and tensor algebras}\label{s_reptype}

\subsection{Separable equivalence}

Many interesting properties of separable equivalence can be found in \cite{Pe}. If $A$ and $B$ are separably equivalent, then  we have, for example, the following:
\begin{itemize}
\item $A$ and $B$ have the same representation type (when $\Bbbk$ is algebraically closed).
\item If $A$ is domestic or of polynomial growth, then so is $B$.
\item For any algebra $C$, the algebras 
$A\otimes_\Bbbk C$ and $B\otimes_\Bbbk C$ are separably equivalent.
\end{itemize}
The main aim of this section is to generalize these facts  to two-sided equivalence.

From \cite[Corollary~6.1]{Ka1} we also deduce that, if $B$ separably divides $A$, then\begin{displaymath}
\mathrm{gl.dim}(B)\geq \mathrm{gl.dim}(A).
\end{displaymath}
In particular, separably equivalent algebras have the same global dimension.

As in \cite[Section~4]{Pe}, symmetric separable equivalence can be extended from algebras to categories as follows: two exact categories $\mathcal{A}$ and $\mathcal{B}$ are symmetrically separably equivalent if there are functors
\begin{displaymath}
\xymatrix{
\mathcal{A}\ar@/^/[r]^F & \mathcal{B}\ar@/^/[l]^G
}
\end{displaymath}
such that $F$ and $G$ are biadjoint, and the identity functors are direct summands of $GF$ and $FG$.

Then two algebras $A$ and $B$ are symmetrically separably equivalent if $A\text{-mod}$ and $B\text{-mod}$ are.

If we have two pairs of symmetrically separably equivalent categories
\begin{displaymath}
\xymatrix{
\mathcal{A}\ar@/^/[r]^F & \mathcal{B}\ar@/^/[l]^G
}
\quad
\xymatrix{
\mathcal{C}\ar@/^/[r]^H & \mathcal{D}\ar@/^/[l]^K
}
\end{displaymath}
then the following is also a symmetrically separable equivalence.
\begin{displaymath}
\xymatrix{
\mathrm{Fun}( \mathcal{A},\mathcal{C}) \ar@/^/[r]^{H\circ - \circ G} & \mathrm{Fun}(\mathcal{B},\mathcal{D}) \ar@/^/[l]^{K\circ -\circ F}
}
\end{displaymath}

Moreover, if $\mathcal{A}'\subseteq \mathcal{A}$ and $\mathcal{B}'\subseteq B'$ are full subcategories 
such that $F\mathcal{A}'\subseteq \mathcal{B}'$ and $G\mathcal{B}'\subseteq \mathcal{A}'$,
then $\mathcal{A}'$ and $\mathcal{B}'$ are symmetrically separably equivalent categories. 
So are $\mathcal{A}/\mathcal{A}'$ and $\mathcal{B}/\mathcal{B}'$.

In particular, if $A$ and $B$ are symmetrically separably equivalent algebras, then 
$\mathrm{Fun}(A\text{-mod} ,\Bbbk\text{-mod})$ and $\mathrm{Fun}( B\text{-mod} ,\Bbbk\text{-mod})$
are symmetrically separably equivalent categories, so that the Auslander algebras of $A$ and $B$ are symmetrically separably equivalent.
Moreover, the stable  categories $A$-\underline{mod} and $B$-\underline{mod} are symmetrically separably equivalent.

\subsection{Representation type}
Throughout this subsection, $\Bbbk$ is an algebraically closed field, and modules are not necessarily finite-dimensional. We are going to generalize \cite[Section~5]{Pe} from separable equivalences to two-sided equivalences. In particular, we shall generalize \cite[Theorem~6]{Pe} to prove that two-sided equivalence preserves representation type. As it turns out, most of the arguments in 
\cite{Pe} also work in the setup of two-sided equivalence, so our arguments
follow closely the original arguments from \cite{Pe}.

Recall that a $\Bbbk$-algebra $A$ is of {\em finite representation type} if there are finitely many isomorphism classes of finite-dimensional indecomposable $A$-modules. If $A$ is not of finite representation type, it is of {\em infinite representation type}. This splits into two main cases.
One says that $A$ is of {\em tame representation type} if, for each $n\in \mathbb{N}$, all but finitely many isomorphism classes of $n$-dimensional indecomposables occur in a finite number of one-parameter families. Finally, $A$ is said to be
of {\em wild representation type} if there are $n$-parameter families of indecomposable $A$-modules for arbitrarily large $n$.

\subsubsection{Generic modules}
If $A$ is a $\Bbbk$-algebra and $M$ a right $A$-module, then $M$ is also a left $\mathrm{End}_A(M)$-module. Following \cite{CB}, we say that the \emph{endolength} of $M$, denoted $\mathrm{end.len.}(M)$, is the length of $M$ as $\mathrm{End}_A(M)$-module. If $M$ has finite endolength, we call it \emph{endofinite}. An endofinite module which has infinite length as $A$-module is called \emph{generic}.

For a positive integer $d$, denote by $g_A(d)$ the number of generic $A$-modules of endolength $d$.

If $A$ is of finite or tame representation type, we define
$\mu_A(n)$ as the smallest integer $m$ such that there are $A$-$\Bbbk [x]$-bimodules $M_1,\ldots ,M_m$ which are free of rank $n$ as $\Bbbk [x]$-modules, and such that all but finitely many indecomposable $A$-modules of dimension $n$ belong to the set
\begin{displaymath}
\{ M_i\otimes \Bbbk [x] /(x-\lambda )\mid \lambda\in \Bbbk,\, i=1,\ldots ,m\} .
\end{displaymath}
The following theorem is proved in \cite[Section~5.6]{CB}.

\begin{theorem}\label{thm_cb}
With notation as above, $\mu_A(n)=\sum_{d|n} g_A(d)$.
\end{theorem}

As is shown in \cite{Pe}, this gives the following characterisation of representation type.

\begin{corollary}
Let $A$ be a $\Bbbk$-algebra.
\begin{enumerate}[(i)]
\item $A$ is of finite representation type if and only if $g_A(d)=0$, for all $d$.
\item $A$ is of tame representation type if and only if $g_A(d)<\infty$, for all $d$, and $g_A(d)>0$ for some $d$.
\item $A$ is of wild representation type if and only if $g_A(d)=\infty$, for some $d$.
\end{enumerate}
\end{corollary}

Moreover, the following result on endofinite decomposability is proved in \cite[Section~4.4.3]{Pr}.

\begin{theorem}\label{thm_prest}
If $M$ is an endofinite module, then there are pairwise non-isomorphic indecomposable endofinite modules $M_1,\ldots ,M_n$, and positive integers $k_1,\ldots ,k_n$ such that
\begin{align*}
M\simeq \bigoplus_{i=1}^n M_i^{\oplus k_i}
\end{align*}
and
\begin{align*}
\mathrm{end.len} (M) = \sum_{i=1}^n \mathrm{end.len} (M_i).
\end{align*}
\end{theorem}

Finally, we have the following analogue of \cite[Lemma~5.1]{Pe} about the endolengths of  tensor products.

\begin{lemma}\label{lemma_endlen}
Let $M$ be an $A$-$B$-bimodule such that ${}_AM$ is finitely generated. If $X_A$ is an endofinite module, then there is a constant $c_M$ such that
\begin{align*}
\mathrm{end.len} (X\otimes_A M)\leq c_M \cdot \mathrm{end.len} (X).
\end{align*}
\end{lemma}

\begin{proof}
Since $M$ is finitely generated as left $A$-module, for some $n$ there is an epimorphism $A^n\to M$. Tensoring with $X$ yields an epimorphism $X^n\to X\otimes_A M$. Therefore the length of $X\otimes_A M$ as an
$\mathrm{End}_A(X)$-module is at most $n\cdot \mathrm{end.len}(X)$.
Now, we have an algebra morphism
\begin{align*}
\mathrm{End}_A(X) &\to \mathrm{End}_B(X\otimes_AM)\\
\varphi \mapsto &\varphi\otimes_A \mathrm{id}_M.
\end{align*}
Hence, any $\mathrm{End}_B(X\otimes_A M)$-module is also an $\mathrm{End}_A(X)$-module. Consequently, any chain of $\mathrm{End}_B(X\otimes_A M)$-submodules
\begin{displaymath}
0=U_0 \subsetneq U_1\subsetneq \ldots \subsetneq U_{r-1} \subsetneq U_r=X\otimes_A M
\end{displaymath}
is also a chain of $\mathrm{End}_A(X)$-modules, implying that $r\leq n\cdot \mathrm{end.len}(X)$. Hence we can take $c_M=n$ and the result follows.
\end{proof}

\subsubsection{Two-sided equivalence preserves representation type}

\begin{theorem}\label{thm_reptype}
Let $\Bbbk$ be an algebraically closed field, and let $A$ and $B$ be $\Bbbk$-algebras such that $A\leq_J B$. Then the following holds.
\begin{enumerate}[(i)]
\item If $A$ is of finite representation type, then $B$ is of finite representation type.
\item If $A$ is of tame representation type, then $B$ is of tame or finite representation type.
\end{enumerate}
In particular, if $A\sim_J B$, then $A$ and $B$ have the same representation type.
\end{theorem}

\begin{proof}
Let ${}_AM_B$ and ${}_BN_A$ be finite-dimensional bimodules  such that $N\otimes_A M \simeq B\oplus Y$ as $B$-$B$-bimodules.

Assume that $A$ is of tame of finite representation type. For each $d$, denote the generic right $A$-modules by ${}^dG_1,\ldots ,{}^dG_{g_A(d)}$.

Let $H$ be an indecomposable generic right $B$-module of endolength $d$. By Lemma~\ref{lemma_endlen}, $H\otimes_B N$ has endolength as most $c_N d$. Then, by Theorem~\ref{thm_prest},
\begin{displaymath}
H\otimes_B N \simeq \bigoplus_{j=1}^{c_Nd} \bigoplus_{i=1}^{g_A(j)} {}^j G_i^{\oplus k_{ij}} \oplus F,
\end{displaymath}
where all $k_{ij}$ are nonnegative integers and $F_A$ is a finite length module.

Since $B$ is a direct summand  of $N\otimes_A M$, it follows that $H$ is a direct summand of $H\otimes_B N\otimes_A M$.
Since $H$ is generic, it has infinite length as $B$-module. Therefore $H$ is not a direct summand of $F\otimes_A M$.
Thus $H$ is a direct summand of some ${}^jG_i\otimes_A M$ with $1\leq j\leq c_Nd$ and $1\leq i\leq g_A(j)$.

What we have shown to this point is that any generic right $B$-module $H$ of endolength $d$ is a direct summand of some $G\otimes_A M$, where $G$ is a generic $A$-module of endolength at most $c_Nd$.
Since $A$ is of finite or tame representation type, the number of generic $A$-modules of each endolength is finite, so the number of generic $B$-modules of each endolength must be finite as well. Moreover, if $A$ is of finite representation type, then there are no generic $A$-modules at all, implying that there are no generic $B$-modules either. This proves both $(i)$ and $(ii)$.

By Drozd's trichotomy theorem, see \cite{Dr}, it also follows that if $A$ is of wild representation type, then so is $B$. This concludes the proof.
\end{proof}

\subsubsection{Tame type}
Tame representation type can be subdivided into different subtypes, given by certain bounds on the number of indecomposables of each finite dimension.

Let $A$ be an algebra of tame representation type. Then $A$ is called  domestic if there is some $N\in \mathbb{N}$ such that $\mu_A(n)\leq N$, for all $n\in \mathbb{N}$.

We say that $A$ is of polynomial growth if there are positive integers $C$ and $d$ such that $\mu_A(n)\leq Cn^d$, for all $n\in \mathbb{N}$.

The following  is a straightforward generalization of \cite[Theorem~7]{Pe}, 
and we refer to \cite{Pe} for the proof.

\begin{theorem}\label{thm_domestic}
Let $\Bbbk$ be an algebraically closed field, and let $A$ and $B$ be finite-dimensional $\Bbbk$-algebras of tame representation type such that $A\sim_J B$. Then the following holds.
\begin{enumerate}[(i)]
\item If $A$ is domestic, then $B$ is domestic.
\item If $A$ is of polynomial growth, then $B$ is of polynomial growth.
\end{enumerate}
\end{theorem}

\subsection{Tensor algebras}
The following result is stated and proved for separable division in \cite[Section~4]{Pe}. It generalizes easily to the two-sided inequality.

\begin{proposition}\label{prop_otimesC}
Let $A$ and $B$ be $\Bbbk$-algebras such that $A\leq_J B$. Then, for any $\Bbbk$-algebra $C$, we have $A\otimes_\Bbbk C\leq_J B\otimes_\Bbbk C$ and $C\otimes_\Bbbk A\leq_J C\otimes_\Bbbk B$.
\end{proposition}

\begin{proof}
We prove the first inequality, the second is analoguous.
Let ${}_AM_B$ and ${}_BN_A$ be such that $N\otimes_A M\simeq B\oplus Y$ as $B$-$B$-bimodules.
Then $N\otimes_\Bbbk C$ is a $B\otimes_\Bbbk C$-$A\otimes_\Bbbk C$-bimodule via
\begin{align*}
(b\otimes c')(n\otimes c)(a\otimes c'')=(bna)\otimes (c'cc'').
\end{align*}
Similarly $M\otimes C$ is an $A\otimes_\Bbbk C$-$B\otimes_\Bbbk C$-bimodule. We have
\begin{align*}
(N\otimes_\Bbbk C) \otimes_{A\otimes C} (M\otimes_\Bbbk C)\simeq (N\otimes_A M )\otimes_\Bbbk C \simeq (B\oplus Y) \otimes_\Bbbk C
\end{align*}
so $B\otimes_\Bbbk C$ is a direct summand of $(N\otimes_\Bbbk C) \otimes_{A\otimes C} (M\otimes_\Bbbk C)$.
\end{proof}

The following corollary of Theorem~\ref{thm_reptype} and Proposition~\ref{prop_otimesC} will be used many times to establish inequivalence of algebras in the proceeding section.

\begin{corollary}\label{cor_enveloping_reptype}
If $A$ and $B$ are $J$-equivalent $\Bbbk$-algebras, then $A^\mathrm{op}\sim_J B^\mathrm{op}$. Moreover, $A\otimes_\Bbbk A^\mathrm{op}$ and $B\otimes_\Bbbk B^\mathrm{op}$ are also $J$-equivalent, and, if the field is algeraically closed, have the same representation type.
\end{corollary}

\begin{proof}
Assume that $A\sim_J B$.
First of all, we note that, if $A$ is a direct summand of ${}_AM\otimes_BN_A$, then $A^\mathrm{op}$ is a direct summand of
\begin{displaymath}
{}_{A^{\mathrm{op}}}N \otimes_{B^\mathrm{op}}M_{A^{\mathrm{op}}} \simeq ({}_AM\otimes_B N_A)^\mathrm{op}.
\end{displaymath}
Hence $A^\mathrm{op}\sim_J B^\mathrm{op}$.
Now, Proposition~\ref{prop_otimesC} gives that
\begin{displaymath}
A\otimes_\Bbbk A^\mathrm{op}\sim_J B\otimes_\Bbbk A^\mathrm{op} \sim_J B\otimes_\Bbbk B^\mathrm{op}.
\end{displaymath}
The rest follows from Theorem~\ref{thm_reptype}.
\end{proof}

\section{Exampels of (in)equivalence}\label{s_examples}

\subsection{Main examples}
Given a finite quiver $Q$, we denote by $R$ the ideal of $\Bbbk Q$ generated by the arrows of $Q$.

Denote by $\mathbb{A}_n$ the uniformly oriented Dynkin quiver of type $A$ on $n$ vertices:
\begin{displaymath}
1\to 2\to \ldots \to n-1 \to n .
\end{displaymath}
Set $A_n:=\Bbbk \mathbb{A}_n/R^2$.

For $n\geq 2$, let  $Q_n$ be the following quiver:
\begin{displaymath}
\xymatrix{
&n\ar[dl]\\
1\ar[r] & \ldots \ar[r] & n-1\ar[ul]
}
\end{displaymath}
For $k\geq 2$, set
\begin{align*}
\Lambda_n^{(k)}:=\Bbbk Q_n/R^k,
\end{align*}
so that $\Lambda_n^{(k)}$ is the path algebra of $Q_n$ modulo the 
relations that any path of length $k$ is zero.

Recall that $\Theta$ denotes the path algebra of the Kronecker quiver,
and $A_3'$ the  path algebra of the following quiver.
\begin{displaymath}
\xymatrix{1\ar[r] & 2 & 3\ar[l]}
\end{displaymath}

Our main examples of two-sided (in)equivalences are summarized here.

\begin{theorem}\label{thm_nakayama}
Let $\Bbbk$ be an algebraically closed field and $n\geq 2$ an integer such that $\mathrm{char}(\Bbbk) \nmid n$. Then the following holds.
\begin{enumerate}[(i)]
\item There is a chain of strict two-sided inequalities
\begin{displaymath}
\Bbbk >_J A_2 >_JA_3 >_J \ldots >_J \Bbbk [x]/ (x^2)>_J \Bbbk [x]/ (x^3)>_J \Bbbk [x]/ (x^4)>_J \Bbbk [x]/ (x^5).
\end{displaymath}
\item For all $k\geq 2$, we have that $\Lambda_n^{(k)}\sim_J \Bbbk [x]/(x^k)$.
\item $\Bbbk[x,y]/(x^2,xy,y^2)\leq_J \Theta <_J \Bbbk [x]/(x^2)$.
\item $\Theta <_J A_3'$. 
\end{enumerate}
\end{theorem}

The remainder of this section is dedicated to the proof of 
Theorem~\ref{thm_nakayama}, split into a number of special cases.

\subsection{Triangular matrix algebras}
For an algebra $A$, denote by $T_n(A)$ the matrix algebra of upper-triangular $n\times n$-matrices with coefficients in $A$. Then
\begin{displaymath}
T_n (A) \simeq A\otimes_\Bbbk \Bbbk \mathbb{A}_n.
\end{displaymath} 

The following is now a special case of Proposition~\ref{prop_otimesC}.
We will used it many times in this section.

\begin{lemma}\label{lemma_Tn}
Fix a positive integer $n$ and two $\Bbbk$-algebras $A$ and $B$.
If $A\leq_J B$, then $T_n(A)\leq_J T_n(B)$. In particular, if $A\sim_J B$, then $T_n(A)\sim_J T_n(B)$.
\end{lemma}

\subsection{Self-injective Nakayama algebras}

First, we prove Theorem~\ref{thm_nakayama}$(ii)$. Indeed, we have an even stronger statement.

\begin{proposition}\label{prop_lambda}
Let $\Bbbk$ be a field and $n\geq 2$ an integer such that $\mathrm{char}(\Bbbk) \nmid n$. 
For each $k\geq 2$, the algebra $\Bbbk [x]/(x^k)$ is symmetrically separably equivalent to $\Lambda_n^{(k)}$.
\end{proposition}

\begin{proof}
Fix $n$ and $k$. Consider the cyclic group $C_n$ of order $n$, and fix a generator $c\in C_n$. Let $C_n$ act on $\Lambda_n^{(k)}$ by counter-clockwise rotation, i.e.
\begin{align*}
\begin{cases}
c\varepsilon_i=\varepsilon_{i+1}\\
c\alpha_{i}=\alpha_{i+1}
\end{cases} 
\end{align*}
with $\varepsilon_{n+1}=\varepsilon_i$ and $\alpha_{n+1}=\alpha_1$.
Clearly this gives a free action on $Q_n$. For $0\leq j\leq k-1$, denote by $a_j$ the sum of all paths of length $j$. Then
\begin{displaymath}
A^{C_n} = \mathrm{span}\{ a_0,a_1,\ldots ,a_{k-1}\} 
\end{displaymath}
and the assignment
\begin{displaymath}
a_j\mapsto x^j
\end{displaymath}
extends to an algebra isomorphism $A^{C_n}\to \Bbbk [x]/(x^k)$. By Theorem~\ref{thm_group}$(ii)$ we have $\Lambda_n^{(k)}\sim_J A^{C_n}$, so it follows that $\Lambda_n^{(k)}\sim_J \Bbbk [x]/(x^k)$.

Moreover, $\Bbbk [x]/(x^k)$ is symmetric and, as $\Lambda_n^{(k)}$-$\Lambda_n^{(k)}$-bimodules, we have
\begin{displaymath}
\mathrm{Hom}_\Bbbk (\Lambda_n^{(k)},\Bbbk ) \simeq (\Lambda_n^{(k)})^{c^{k-1}},
\end{displaymath}
where $(\Lambda_n^{(k)})^{c^{k-1}}$ is the regular $\Lambda_n^{(k)}$-$\Lambda_n^{(k)}$-bimodule with the right action twisted by the action of $c^{k-1}\in C_n$. By Theorem~\ref{thm_group}$(iii)$, it follows that $\Lambda_n^{(k)}$ is symmetrically separably equivalent to $\Bbbk [x]/(x^k)$.
\end{proof}

\begin{remark}
\begin{enumerate}[(i)]
\item The algebra $\Bbbk [x]/(x^2)$ is symmetric, whereas $\Lambda_k^{(2)}$ is not symmetric, for any $k\geq 2$. Thus Proposition~\ref{prop_lambda} provides an example showing that symmetrically separable equivalence does not necessarily preserve the property of being symmetric as algebra.
\item Proposition~\ref{prop_lambda} also follows from Theorem~\ref{thm_skew}. As seen in \cite[Section~2.4]{RR}, $\Lambda_n^{(k)}$ is Morita equivalent to the skew group algebra $\Bbbk [x]/(x^k)\ast C_n$, where $C_n$ acts as $x\mapsto \zeta x$ for some primitive $n$'th root of unity $\zeta$.
\end{enumerate}
\end{remark}

In \cite[Theorem~8]{Pe}, it is shown that the truncated polynomial algebras $\Bbbk [x]/(x^k)$ and $\Bbbk [x]/(x^l)$ are not separably equivalent, for $k\neq l$ when $k=1,\ldots ,6$. Parts of this proof can be modified to $J$-inequivalence, and we obtain the following.

\begin{proposition}\label{prop_345}
Let $\Bbbk$ be an algebraically closed field. Then the following 
strict inequalities hold.
\begin{displaymath}
\Bbbk[x] /(x^3) >_J \Bbbk[x] /(x^4) >_J\Bbbk[x] /(x^5).
\end{displaymath}
\end{proposition}

\begin{proof}
As $\Bbbk [x]/(x^k)$ is a quotient of $\Bbbk [x]/(x^l)$, for $k<l$, the statement with non-strict inequalities follows from Proposition~\ref{prop_surj}. We shall now prove that we do not have $J$-equivalence of these algebras.

We shall use Lemma~\ref{lemma_Tn}: if $A$ and $B$ are $J$-equivalent, then so are $T_2(A)$ and $T_2(B)$.
The representation type of $T_2 \big( \Bbbk [x]/(x^k)\big)$ depends on $k$ as follows, cf. \cite{LS1} (see also the proof of Theorem~8 in \cite{Pe}):
\begin{center}
\begin{tabular}{c|r}
Representation type &  $k$\\
\hline
Finite & $<4$\\
Tame & 4\\
Wild &  $>4$
\end{tabular}
\end{center}
Therefore, from Lemma~\ref{lemma_Tn} it follows that $\Bbbk[x] /(x^3) >_J \Bbbk[x] /(x^4) >_J\Bbbk[x] /(x^5)$.
\end{proof}

\subsection{Uniformly oriented $\mathbf{A_n}$-quivers with radical square zero}

Recall that $A_n=\Bbbk \mathbb{A}_n/R^2$.

\begin{proposition}\label{prop_An}
Let $\Bbbk$ be an algebraically clsoed field. Then we have the following chain of strict $J$-inequalities:
\begin{displaymath}
A_2>_J A_3>_J A_4>_J \ldots 
\end{displaymath}
\end{proposition}

\begin{proof}
Since $A_{n}$ is a quotient of $A_{n+1}$ for $n\geq 2$, it follows from Proposition~\ref{prop_surj} that $A_n\geq_J A_{n+1}$.
Now, fix some $m,n\geq 2$. We use the notions of $A_m$-$A_n$-bimodules and left $A_n\otimes_\Bbbk A_n^\mathrm{op}$-modules interchangeably. 
The algebra $A_m\otimes_\Bbbk A_n^\mathrm{op}$ is isomorphic to the path algebra of the quiver
\begin{displaymath}
\xymatrix{
1|1\ar[d] & 1|2\ar[d]\ar[l] & \ldots \ar[l]& 1|n\ar[d]\ar[l]\\
2|1\ar[d] & 2|2\ar[d]\ar[l] & \ldots \ar[l] & 2|n\ar[d]\ar[l]\\
\vdots\ar[d] & \vdots \ar[d] & \ddots & \vdots\ar[d]\\
m|1 & m|2 \ar[l]& \ldots \ar[l] & m|n\ar[l]
}
\end{displaymath}
modulo the relations that all squares commute, and the composition of any two horizontal or any two vertical arrows is zero.
This algebra is special biserial and of finite representation type. Each indecomposable module corresponds to either a commuting square, or to a walk of alternating steps downwards (along the vertical arrows) and to the right (against the horizontal arrows); each visited node contains a copy of $\Bbbk$, and each arrow in the walk is the identity on $\Bbbk$. See \cite{BR,WW} for the classification, and \cite{MZ} for an explicit description of the indecomposables in case $n=m$.

In particular, we observe that, for any indecomposable $A_m$-$A_n$-bimodule $M$,
\begin{displaymath}
\dim \mathrm{top} (M)\leq \min \{ m,n\} .
\end{displaymath}
Now Lemma~\ref{lemma_top} implies that if $m>n$, then $A_m\not\geq_J A_n$.
\end{proof}

\subsection{More on Nakayama algebras}\label{ss_more_nakayama}

\begin{lemma}\label{lemma_D}
Let $\Bbbk$ be an algebraically closed field and $B$ a connected $\Bbbk$-algebra. Then the following holds.
\begin{enumerate}[(i)]
\item If $B\sim_J \Bbbk [x]/(x^2)$, then $B$ is Morita equivalent to either $\Bbbk [x]/(x^2)$ or $\Lambda_n^{(2)}$, for some $n\geq 2$.
\item If $B\sim_J A_m$, for some positive integer $m$, then $B$ and $A_m$ are Morita equivalent.
\end{enumerate}
\end{lemma}

\begin{proof}
Since we are interested in algebras only up to Morita equivalence, we
may assume that $B$ is basic.
According to \cite[Theorem~6.1]{LS2}, the algebra $T_3(B)$ is of finite representation type if  and only if $B$ is a radical square zero Nakayama algebra. 

Since $\Bbbk [x]/(x^2)$ and $A_m$, for $m\geq 1$, are radical square zero Nakayama algebras, it follows that they can only be two-sided equivalent to other radical square zero Nakayama algebras.

However, $\Bbbk [x]/(x^2)\otimes_\Bbbk \Bbbk [x]/(x^2)$ has tame representation type, whereas $A_m\otimes_\Bbbk A_m^\mathrm{op}$ has finite representation type, for all $n$. Thus Corollary~\ref{cor_enveloping_reptype} implies that $\Bbbk [x]/(x^2)$ and $A_m$ are inequivalent for all $m$.
\end{proof}

\begin{lemma}\label{lemma_DA}
Let $\Bbbk$ be a field. Then, for any positive integer $m$, and any $k\geq n\geq 2$, $\Bbbk [x]/(x^k)\leq_J \Bbbk \mathbb{A}_m/R^n$.
\end{lemma}

\begin{proof}
Fix some $l\geq m$ such that $\mathrm{char}\, (\Bbbk )\nmid l$. 
By Proposition~\ref{prop_lambda}, $\Bbbk [x]/(x^k)\sim_J \Lambda_l^{(k)}$.
Since there is a surjective algebra morphism $\Lambda_l^{(k)}\twoheadrightarrow \Bbbk \mathbb{A}_m/R^n$, the claim follows from Proposition~\ref{prop_surj}.
\end{proof}

Theorem~\ref{thm_nakayama}$(i)$ now follows from Propositions \ref{prop_345} and \ref{prop_An}, together with Lemmata~\ref{lemma_D} and \ref{lemma_DA}.

\subsection{Kronecker algebra}
In this section, we shall prove Theorem~\ref{thm_nakayama}$(iii)$ and $(iv)$.

\subsubsection{Polynomial algebras}
It is immediate from Proposition~\ref{prop_surj} that
$\Theta\leq_J A_2$.
However, if the field is algebraically closed, then $\Theta $ is of tame representation type, whereas $A_m$ is of finite type, for all $m$. This indicates that the Kronecker algebra and $A_2$ should not be ``close`` in the two-sided order.
Indeed, the Kronecker algebra is smaller than the dual numbers $D:=\Bbbk [x]/(x^2)$  in the two-sided preorder.

\begin{proposition}\label{prop_kronecker}
For any field $\Bbbk$ we have that
$\Theta\leq_J \Bbbk [x]/(x^2)$.
If $\Bbbk$ is algebraically closed, then the inequality is strict.
\end{proposition}

\begin{proof}
The proof of the first part is completely explicit: we define two bimodules and compute their tensor product over $\Theta$ to verify that it is isomorphic to $D=\Bbbk [x]/(x^2)$ as a $D$-$D$-bimodule.

Let $\varepsilon_1,\varepsilon_2$ be the paths of length zero 
at the nodes 1 and 2 of the Kronecker quiver, respectively.

Consider now the $D$-$\Theta$-bimodule $M$ which, as a representation of the Kronecker quiver, has dimension vector $(2,2)$, and  such that the right action of $\alpha$ is given by $I_2$, and of $\beta$ by the transpose $2\times 2$-Jordan block with eigenvalue 1, that is
\begin{align*}
J_2(1)^t =\begin{bmatrix}
1&0\\ 1&1
\end{bmatrix}
\end{align*}
 Moreover, on the copy of $\Bbbk^2$ at each vertex, the left action of $D$ is determined by the $x$ action via $\begin{bmatrix}
0 & 0\\ 1& 0
\end{bmatrix}$. Putting everything together, we get:
\begin{displaymath}
\xymatrix{
\Bbbk^2 \ar@(ul,dl)[]_{\begin{bmatrix}0&0\\ 1&0 \end{bmatrix}}& \ar@/_/[l]_{I_2} \ar@/^/[l]^{J_2(1)^t}  \Bbbk^2 \ar@(ur,dr)[]^{\begin{bmatrix}0&0\\ 1&0 \end{bmatrix}}
}
\end{displaymath}
Set $N=\mathrm{Hom}_D(M,D)$. Then $N$ is a $\Theta$-$D$-bimodule, with the action given as follows:
\begin{displaymath}
\xymatrix{
\Bbbk^2 \ar@(ul,dl)[]_{\begin{bmatrix}0&0\\ 1&0 \end{bmatrix}} \ar@/^/[r]^{I_2} \ar@/_/[r]_{J_2(1)^t}&  \Bbbk^2 \ar@(ur,dr)[]^{\begin{bmatrix}0&0\\ 1&0 \end{bmatrix}}
}
\end{displaymath}
We claim that $M\otimes_\Theta N\simeq D$, as $D$-$D$-bimodules.

In $M$, fix bases $\{ m_1,m_2\}$ and $\{ m_3,m_4\}$ of $M\varepsilon_1$ and $M\varepsilon_2$, respectively. Similarly, in $N$, fix bases $\{ n_1,n_2\}$ and $\{ n_3,n_4\}$ of $\varepsilon_1N$ and $\varepsilon_2N$.

Now consider the tensor product $M\otimes_\Theta N$. 
We shall find a basis of  this tensor product by using the definition: $M\otimes_\Theta N$ is the quotient of  $M\otimes_\Bbbk N$ by the subbimodule generated by
\begin{displaymath}
mb\otimes n=m\otimes bn,\quad m\in M,\, n\in N,\,b\in \Theta.
\end{displaymath}
First of all, we have
\begin{displaymath}
m_1\otimes n_3= m_1\varepsilon_1\otimes n_3=m_1\otimes \varepsilon_1 n_3=0.
\end{displaymath}
In the same way, we deduce that, for all $i\in \{ 1,2\}$ and $j\in \{ 3,4\}$, we have
\begin{align*}
m_i\otimes n_j=0=m_j\otimes n_i.
\end{align*}
Using the action of $\alpha$, we find
\begin{align}
m_1\otimes n_1= m_3\alpha \otimes n_1&=m_3\otimes \alpha n_1= m_3\otimes n_3 \label{eq:kron1} \\
m_1\otimes n_2&=m_3\otimes n_4 \label{eq:kron2} \\
m_2\otimes n_1&= m_4\otimes n_3 \label{eq:kron3} \\
m_2\otimes n_2 &= m_4\otimes n_4. \label{eq:kron4}
\end{align}
Moreover, using $\beta$ yields
\begin{align}
m_1\otimes n_1 = (m_3-m_4)\beta \otimes n_1 = (m_3-m_4)\otimes \beta n_1 = (m_3-m_4)\otimes (n_3+n_4) \label{eq:kron5}
\end{align}
and
\begin{align}
m_1\otimes n_2 = (m_3-m_4)\beta \otimes n_2 = (m_3-m_4)\otimes n_4. \label{eq:kron6}
\end{align}
Combining \eqref{eq:kron2} and \eqref{eq:kron6}, yields that $m_4\otimes n_4=0$. This, together with \eqref{eq:kron1} and \eqref{eq:kron5}, implies that
\begin{displaymath}
m_3\otimes n_4=m_4\otimes n_3.
\end{displaymath}
What remains are the two linearly independent basis elements
\begin{align*}
m_1\otimes n_1 &= m_3\otimes n_3\\
m_1\otimes n_2=m_3\otimes n_4&=m_4\otimes n_3=m_2\otimes n_1.
\end{align*}
The linear map $M\otimes_\Theta N\to D$, defined via
\begin{align*}
\begin{cases}
m_1\otimes n_1\mapsto 1\\
m_1\otimes n_2\mapsto x,
\end{cases}
\end{align*}
is an isomorphism of $D$-$D$-bimodules, as
\begin{displaymath}
xm_1\otimes n_1=m_2\otimes n_1 = m_1\otimes n_2 = m_1\otimes n_1x.
\end{displaymath}
This proves that $\Theta \leq_J D$.

If the field is algebraically closed, the strict inequality follows from Lemma~\ref{lemma_D}.
\end{proof}

\begin{remark}
It is clear that $M$ above is left projective, so that $(M\otimes_\Theta -,N\otimes_D-)$ form an adjoint pair of functors. 
The $D$-$D$-bimodule isomorphism $M\otimes_{\Theta} N=M\otimes_\Theta \mathrm{Hom}_{D}(M,D)\to D$ is just the evaluation.
\end{remark}

\begin{proposition}
If $\Bbbk$ is a perfect field, then $\Theta\geq_J \Bbbk [x,y]/(x^2,xy,y^2)$.
\end{proposition}

\begin{proof}
$\Theta$ has a subalgebra isomorphic to $\Bbbk [x,y]/(x^2,xy,y^2)$ spanned by $\{1,\alpha ,\beta\}$. Its radical is $(\alpha ,\beta )=\mathrm{rad}(\Theta )$. By \cite[Proposition~1.2]{Ka3}, the extension is separable, implying the desired result.
\end{proof}

\subsubsection{Non-uniformly oriented $A_3$}
For $n\geq 2$, denote by $Q_n'$ the following quiver.
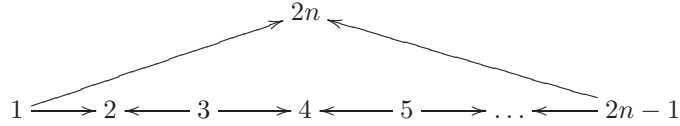
\begin{displaymath}
\xymatrix{
&&& 2n\\
1\ar[r]\ar[urrr] & 2 & 3\ar[l]\ar[r] & 4 & 5\ar[r]\ar[l] & \ldots & 2n-1 \ar[l]\ar[ulll]
}
\end{displaymath}

\begin{proposition}
Fix an integers $n\geq 2$.
\begin{enumerate}[(i)]
\item If $\mathrm{char}(\Bbbk )\neq 2$, then $A_3'$ and $( A_3')^\mathrm{op}$ are symmetrically separably equivalent.
\item If $\mathrm{char} \, (\Bbbk )$ does not divide $n$, then $\Theta$ and $\Bbbk Q_n'$ are symmetrically separably equivalent.
\item $\Theta \leq_J A_3'$, and if $\Bbbk$ is algebraically closed the inequality is strict.
\end{enumerate}
\end{proposition}

\begin{proof}
For parts (i) and (ii) we use Theorem~\ref{thm_skew} and  examples from \cite[Section~2.4]{RR}. 
Indeed, if $C_2$ acts on $(A_3')^\mathrm{op}$ by reflection in the vertical line through the node 2, then $(A_3')^\mathrm{op}\ast C_2$ is Morita equivalent to $A_3'$ (Example (b)).
Likewise, as in Example (j), take some primitive $n$'th root of unity and let $C_n=\left\langle c\right\rangle$ act on $\Theta$ by $c\alpha =\alpha$, $c\beta =\zeta \beta$. Then $\Theta \ast C_n$ is Morita equivalent to $\Bbbk Q_n'$.

Part (iii) is already known from Example~\ref{ex_sep}$(ii)$ in case $\Bbbk$ is perfect.
If not, take some $n$ such that $\mathrm{char}\, (\Bbbk )\nmid n$. Then $\Theta \sim_J \Bbbk Q_n'$ by (ii). Since $A_3'$ is a quotient of $\Bbbk Q_n'$ the non-strict inequality follows from Proposition~\ref{prop_surj}.
If the field is algebraically closed, $\Theta$ and $A_3'$ cannot be $J$-equivalent as they have different representation types, so the statement follows.
\end{proof}

\begin{remark}
The fact that $\Theta \sim_J \Bbbk Q_n'$ if $\mathrm{char} \, (\Bbbk )\nmid n$ follows also from Theorem~\ref{thm_group}, since $\Theta$ is isomorphic to the algebra of invariants of $\Bbbk Q_n'$ under the obvious action of $C_n$. However, since $\Theta$ is not symmetric we cannot invoke Theorem~9$(iii)$, so to conclude symmetrically separable equivalence we need the approach of skew group algebras.
\end{remark}

\subsubsection{Additional skew group examples}

We note that the symmetric separable equivalence of $A_3'$ and its opposite algebra is a special case of two more general skew group algebra constructions.

By Example (b) in \cite[Section~2.4]{RR}, the path aglebras of the following two quivers are $J$-equivalent when $\mathrm{char}\, (\Bbbk )\neq 2$.
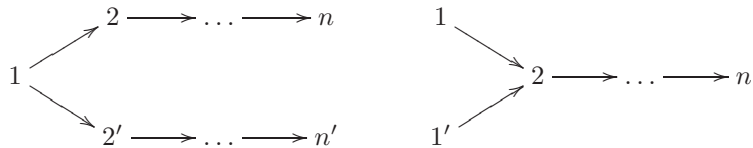
\begin{align*}
\xymatrix@R=10pt{
 & 2\ar[r] & \ldots \ar[r] & n\\
1\ar[ur] \ar[dr]\\
 & 2' \ar[r] & \ldots \ar[r] & n'
}\hspace{1cm}
\xymatrix@R=10pt{
1\ar[dr]\\
& 2\ar[r] & \ldots \ar[r] & n\\
1'\ar[ur]
}
\end{align*}

In a similar way, let $A$ be the path algebra of the following quiver.
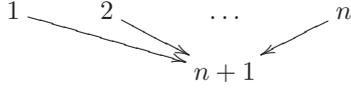
\begin{displaymath}
\xymatrix@R=10pt{
1\ar[drr] & 2\ar[dr] & \ldots & n\ar[dl]\\
&&n+1
}
\end{displaymath}
If $C_n$ acts in the obious way, then $A\ast C_n$ is Morita equivalent to $A^\mathrm{op}$, so $A\sim_J A^\mathrm{op}$ if $\mathrm{char}\, (\Bbbk )\nmid n$.

In the case $n=4$ we get in the last example the path algebra of the four subspace quiver $T_4$ and its opposite. These appear in \cite[Section~2.4]{RR} several times, and we obtain the following.

\begin{proposition}
Let $\Bbbk$ be an algebraically closed field such that $\mathrm{char} \, (\Bbbk)\neq 2,3$.
Then the path algebras of the following quivers are symmetrically separably equivalent.
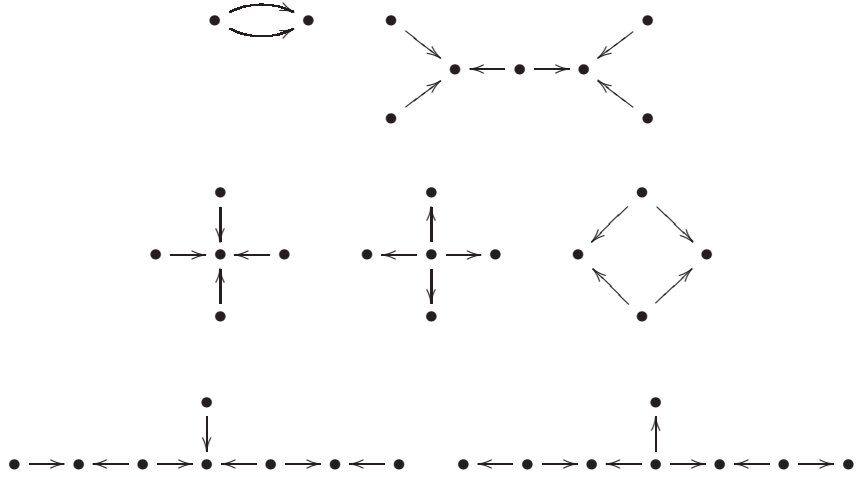
\begin{displaymath}
\xymatrix{
\bullet \ar@/^/[r]\ar@/_/[r] & \bullet
}
\hspace{20pt}
\xymatrix@R=8pt@C=13pt{
\bullet \ar[dr] &&&& \bullet\ar[dl]\\
& \bullet & \bullet \ar[l]\ar[r] & \bullet\\
\bullet \ar[ur] &&&& \bullet \ar[ul]
}
\end{displaymath}

\begin{displaymath}
\xymatrix@R=13pt@C=13pt{
& \bullet \ar[d]\\
\bullet \ar[r] &\bullet & \bullet \ar[l]\\
& \bullet \ar[u]
}
\hspace{20pt}
\xymatrix@R=13pt@C=13pt{
& \bullet \\
\bullet & \bullet \ar[l]\ar[r]\ar[u]\ar[d] & \bullet\\
&\bullet
}
\hspace{20pt}
\xymatrix@R=13pt@C=13pt{
& \bullet \ar[dl]\ar[dr]\\
\bullet && \bullet\\
& \bullet \ar[ul]\ar[ur]
}
\end{displaymath}

\begin{displaymath}
\xymatrix@R=13pt@C=13pt{
&&& \bullet\ar[d]\\
\bullet\ar[r] & \bullet & \bullet\ar[l]\ar[r] & \bullet & \bullet\ar[l]\ar[r] & \bullet & \bullet\ar[l]
}
\hspace{13pt}
\xymatrix@R=13pt@C=13pt{
&&& \bullet\\
\bullet & \bullet \ar[l]\ar[r] & \bullet & \bullet\ar[l]\ar[r] \ar[u] & \bullet & \bullet \ar[l]\ar[r] & \bullet
}
\end{displaymath}
\end{proposition}

\begin{proof}
Except for the last two quivers, of affine Dynkin type $\tilde{E}_7$, all examples are already discussed here or in \cite{RR}. For them, $\mathrm{char}\, (\Bbbk )\neq 2$ is sufficient.

Now, let $S_n$ act on the path algebra of
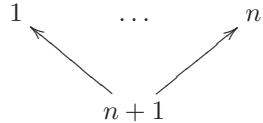
\begin{displaymath}
\xymatrix{
1 & \ldots & n\\
& n+1\ar[ur] \ar[ul]
}
\end{displaymath}
by permuting the nodes $1,\ldots ,n$.
By \cite[Theorem~1]{De} we obtain, if $\mathrm{char}\, (\Bbbk )\nmid n!$, that the skew group algebra is Morita equivalent to the path algebra of the quiver given as follows: the nodes are integer partitions of $n$ and $n-1$. If $\lambda \vdash n$ and $\mu \vdash n-1$, then there is an arrow from $\lambda$ to $\mu$ if and only if  the Young diagram of shape $\mu$ can be included in that of shape $\mu$. In other words, this quiver is the rows $n$ and $n-1$ in the Young lattice, with arrows pointing from $n$ to $n-1$.
In particular, when $n=4$ we have that  $T_4\ast S_4$ is Morita equivalent to the path algebra of the following quiver.
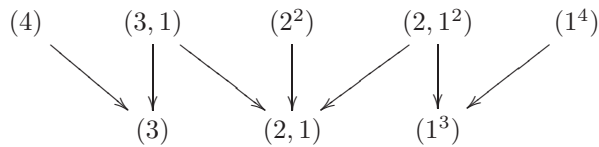
\begin{displaymath}
\xymatrix{
(4)\ar[dr] & (3,1)\ar[d]\ar[dr] & (2^2)\ar[d] & (2,1^2)\ar[d]\ar[dl] & (1^4)\ar[dl]\\
& (3) & (2,1) & (1^3)
}
\end{displaymath}
Reversing the direction of the arrows in the original quiver does the same in the quiver obtained via the skew group algebra construction, so we are done.
\end{proof}

\subsection{Other examples for Nakayama algebras}
We have thus far used the representation types of enveloping algebras and small triangular matrix algebras to prove the inequivalence of  various algebras.
The work \cite{LS1} of Leszczynski and Skowronski on the representation types of tensor products of algebras can provide many more examples. We conclude this section with some additional examples of Nakayama algebras.

Denote by $C$ the path algebra of $\mathbb{A}_4$ modulo the ideal generated by the path from 2 to 4.
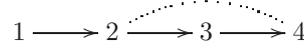
\begin{displaymath}
\xymatrix{
1\ar[r] & 2\ar[r] \ar@/^1pc/@{..}[rr] & 3\ar[r] & 4 
}
\end{displaymath}
This algebra is of interest as it is one of the few examples of algebras with tame enveloping algebras which is not a radical square zero Nakayama algebra.
We could exchange it with its opposite in the following result.

\begin{proposition}
Let $\Bbbk$ be an algebraically closed field. Then
\begin{align*}
\Bbbk \mathbb{A}_6<_J
\Bbbk \mathbb{A}_5 \leq_J \Bbbk \mathbb{A}_5/R^4 <_J\Bbbk \mathbb{A}_4 <_J C\leq_J \Bbbk \mathbb{A}_3<A_3
\end{align*}
and
\begin{align*}
\Bbbk \mathbb{A}_6<_J
\Bbbk \mathbb{A}_6/R^4<_J
\Bbbk \mathbb{A}_5/R^4
\end{align*}
\end{proposition}

\begin{proof}
The non-strict inequalities follows from Proposition~\ref{prop_surj}. 
We consider the enveloping algebras $A\otimes_\Bbbk A^\mathrm{op}$, and the algebras $T_2(A)$, for the relevant algebras $A$.
From \cite[Theorems~7.1,7.2]{LS2} and \cite[Theorems~1,4]{LS1} resepctively, we deduce the following.
\begin{center}
\begin{tabular}{c|c|c}
$A$ & Representation type of $A\otimes A^\mathrm{op}$ & Representation type of $T_2(A)$\\
\hline
$A_3$ & Finite & Finite\\
\hline
$A_4$ & Finite &  Finite\\
\hline
$C$ & Tame & Finite\\
\hline
$\Bbbk \mathbb{A}_3$ & Tame & Finite\\
\hline
$\Bbbk \mathbb{A}_4$ & Wild & Finite\\
\hline
$\Bbbk \mathbb{A}_5/R^4$ & Wild & Tame\\
\hline
$\Bbbk \mathbb{A}_5$ & Wild & Tame\\
\hline
$\Bbbk \mathbb{A}_6$ & Wild & Wild\\
\end{tabular}
\end{center}
The first chain of inequivalences follow from Corollary~\ref{cor_enveloping_reptype} and Lemma~\ref{lemma_Tn}.

Moreover, while both $T_2(\Bbbk \mathbb{A}_6/R^4)$ and $T_2(\Bbbk \mathbb{A}_5/R^4)$ are of tame type, the former is of polynomial growth, whereas the latter is domestic. By Theorem~\ref{thm_domestic} and Lemma~\ref{lemma_Tn}, the two are not $J$-equivalent.
\end{proof}

Next, we shall consider Nakayama algebras that lie ''inbetween'' radical square zero and radical qube zero, i.e. such that all paths of length 3 are zero, and some, but not all, paths of length 2 are zero.
\begin{proposition}
Fix $n\geq 3$. Let $I$ be an ideal in $\Bbbk Q_n$ such that $R^3\subsetneq I\subsetneq R^2$. Then $\Bbbk [x]/(x^2)>_J \Bbbk Q_n /I >_J \Bbbk [x]/(x^3)$.
\end{proposition}

\begin{proof}
For any $n$, we have surjective algebra homomorphisms 
\begin{displaymath}
\Lambda_n^{(3)}\twoheadrightarrow \Bbbk Q_n/I \twoheadrightarrow \Lambda_n^{(2)},
\end{displaymath}
so Proposition~\ref{prop_surj} and Proposition~\ref{prop_lambda} implies the non-stric inequialities.

The statement $\Bbbk [x]/(x^2)>_J \Bbbk Q_n /I$ follows from Lemma~\ref{lemma_D}.

From \cite[Theorem~6.1,Theorem~6.2]{LS2} we deduce that $T_3\big( \Bbbk Q_n /I\big)$ is of tame representation type, whereas $T_3\big( \Bbbk [x]/(x^3)\big)$ is of wild representation type. 
The statement follows from Lemma~\ref{lemma_Tn} and Theorem~\ref{thm_reptype}.
\end{proof}

We conclude with an example utilizing the categorical results from \cite[Section~4]{Pe}.
Fix an integer $n$ and let $I_n$ be the ideal of $\Bbbk Q_{2n}$ generated by all paths between odd nodes.

\begin{corollary}
\begin{enumerate}[(i)]
\item For all integers $m$ and $n$ not divisible by $\mathrm{char}\, (\Bbbk )$, the algebras $\Bbbk Q_{2m}/I_m$ and $\Bbbk Q_{2n}/I_n$ are symmetrically separably equivalent.
\item If $\Bbbk$ is algebraically closed, then $\Bbbk [x]/(x^2)>_J \Bbbk Q_{2n} /I_n>_J \Bbbk [x]/(x^3)$.
\end{enumerate}
\end{corollary}

\begin{proof}
Part $(i)$ follows from \cite[Theorem~2(b)]{Pe}, which, as a special case, tells us that if $A$ and $B$ are symmetrically seaprably equivalent, then so are their Auslander algebras.
We see that $\Bbbk Q_2/I_1$ is the Auslander algebra of the dual numbers, and for $n>1$,
$\Bbbk Q_{2n}/I_n$ is the  Auslander algebra of $\Lambda_n^{(2)}$.
As the dual numbers and $\Lambda_n^{(2)}$ are symmetrically separably equivalent whenever $\mathrm{char}\, (\Bbbk )\nmid n$ by Proposition~\ref{prop_lambda}, the claim follows.

Part $(ii)$ is a special case of the previous proposition.
\end{proof}

\section{Left-right projective bimodules}
Left-right projective bimodules are in themselves an interesting object of study, and they play the main role in stable equivalence of Morita type and separable equivalence.
Despite this, not much is known about the category $A$-lrproj-$B$ of left-right projective $A$-$B$-bimodules in the general case.

To our knowledge, the most detailed investigation is in the case when $A$ and $B$ are both self-injective and can be found in \cite{Po1}. In this case, $A\text{-lrproj-}B$ is of finte type if and only if any left-right projective $A$-$B$-bimodule is projective.
Here we consider the case when one of the algebras is self-injective and the other is directed.
\begin{theorem}\label{thm_lrproj}
Let $A=\Bbbk Q/I$ be a directed algebra (in the sense that $Q$ is acylic), and $B$ a self-injective algebra.
Then any left-right projective $A$-$B$-bimodule is projective.
\end{theorem}

\begin{proof}
We assume that the labelling of the nodes of $Q$ are such that if there is an arrow $\alpha :i\to j$, then $i<j$.
Take some left-right projective $A$-$B$-bimodule $M$.
Set $M_i=\varepsilon_iM$ as the value of ${}_AM$ at the node $i\in Q_0$. 
Moreover, write
\begin{align*}
{}_AM \simeq \bigoplus_{i\in Q_0} A\varepsilon_i^{\oplus d_i} .
\end{align*}

We have $M_iB =(\varepsilon_i M)B=\varepsilon_i (MB)$,
so $M=\bigoplus_{i\in Q_0}M_i$ is a decomposition of $M_B$ into a direct sum of submodules.
In particular, each $M_i$ is a projective $B$-module.

Let $J_0$ be the set of nodes in $i\in Q_0$ such $A\varepsilon_i$ is a direct summand of ${}_AM$, and the subquotient $L_i$ occurs in ${}_AM$ only in this summand, i.e. $[M:L_i]=d_i>0$.
Note that $J_0\neq \emptyset$ since it contains
\begin{displaymath}
i_0=\min \{ i\in Q_0\mid A\varepsilon_i\ \text{is a direct summand of}\ {}_AM\}.
\end{displaymath}
For each $i\in J_0$, fix a proejctive right $B$-module $P_i$ such that $(M_i)_B\simeq P_i$. Considering the left $A$-module structure, $M_i$ is the top of $A\varepsilon_i$, and it generates a projective  subbimodule $A\varepsilon_i\otimes_\Bbbk P_i$, which is by construction a direct summand as a left $A$-module. Moreover, as $B$ is self-injective, $P_i$ is also injective. Hence it is a direct summand also as a right $B$-module.
Set
\begin{displaymath}
Z=\bigoplus_{i\in J_0} A\varepsilon_i\otimes_\Bbbk P_i .
\end{displaymath}
If $Z=M$ we are done. Else, set
\begin{displaymath}
j_0=\min \{ i\in Q_0\setminus J_0 \mid A\varepsilon_i\ \text{is a direct summand of}\ {}_AM\}.
\end{displaymath}
Then, considering the left $A$-module structure,
\begin{displaymath}
M_{j_0} = \varepsilon_{j_0} Z \oplus \mathrm{top} (A\varepsilon_{j_0}^{\oplus d_{j_0}}) .
\end{displaymath}
On the other hand, $M_{j_0}$ is a $B$-submodule, so for some $B$-module $U$,
\begin{align*}
M_{j_0} = \varepsilon_{j_0} Z \oplus U.
\end{align*}

We shall prove that $U=\mathrm{top}(A\varepsilon_{j_0}^{\oplus d_{j_0}})$.

Assume towards contadiction that there is some $u\in U$  such that $Au \in Z$. Fix some $a\in A$ such that $au\in Z\setminus \{ 0\}$.
Write $u=z+v$ for some $z\in (Z\cap M_{j_0})$ and $v\in \mathrm{top}(A\varepsilon_{j_0}^{\oplus d_{j_0}})$.
Now
\begin{displaymath}
au=az+av \in Z .
\end{displaymath}
However, $az\in Z$, and $av\in A\varepsilon_{j_0}^{d_{j_0}}$.
Since $A\varepsilon_{j_0}^{\oplus d_{j_0}}\cap Z=\{ 0\}$ it follows that $av=0$.
But $v$ is in the top of a projective $A$-module, so $av=0$ implies $aM_{j_0}=0$.
In particular $az=0$, so that $au=0$, a contradiction.

By the same reasoning as above, using that $B$ is self-injective, we obtain that $M_{j_0}$ generates a subbimodule $A\varepsilon_{j_0}\otimes_\Bbbk U$, and that $U$ is a projective right $B$-module. Setting
\begin{displaymath}
Z_1=Z\oplus A\varepsilon_{j_0}\otimes_\Bbbk U
\end{displaymath}
we obtain that $Z_1$ is a projective subbimodule which is a direct summand both as left and right module.
We can now repeat the above argument: either $M=Z_1$, or we set 
\begin{align*}
J_1&= J_0\cup \{ j_0\}\\
j_1&= \min \{ i\in Q_0 \setminus J_1 \mid A\varepsilon_i\ \text{is a direct summand of}\ {}_AM\}.
\end{align*}
Since $Q_0$ is finite, this will eventually stop, and we will have written ${}_AM_B$ as a direct sum of projective summands.
\end{proof}

\begin{corollary}\label{cor_not_seprel}
Let $A$ and $B$ be non-semisimple $\Bbbk$-algebras.
If $A$ is directed and $B$ is self-injective, then $A$ and  $B$ are not separably related.
\end{corollary}

In particular, Corollary~\ref{cor_not_seprel} implies that the two-sided inequalities $\Theta \leq_J \Bbbk [x]/(x^2)$  and $\Bbbk [x]/(x^n)\leq_J \Bbbk \mathbb{A}_m/R^n$ do not extend to separable division in the sense of Peacock.

\section{Final remarks}
Let $\ell \ell$ denote Loewy length.

\begin{conjecture}
Let $A$ and $B$ be finite-dimensional $\Bbbk$-algebras.
If $A\geq_J B$, then $\ell \ell (A)\leq \ell \ell (B)$.
\end{conjecture}

The conjecture seems to us intuitive. Interpreting Loewy length as 
nilpotency degree of the radical, it seems unlikely that an algebra with higher Loewy length would factor via the tensor product over an algebra with lower Loewy length.
Moreover, it is inherent to the situation of quotients (Proposition~\ref{prop_surj}) and, as already noted, automatic in the situation of skew group algebras.

There is a close connection between the constructions involving group actions from Section~4, i.e. the subalgebra of invariants and the skew group algebra.
The examples in this paper given by invariants under group actions could also be obtained using skew group algebras.
Moreover, we know from \cite[Corollary~5.2]{RR} if $G$ is abelian, then $A$ is Morita equivalent to $(A\ast G)\ast \Xi$, where $\Xi$ is the group of characters of $G$.
An interesting question is therefore the following: if $A$ is an algebra and $G$ a group acting on $A$, can we define  a group action of some  group $H$ on the skew group algebra $A\ast G$ such that the invariant subalgebra $(A\ast G)^H$ is Morita equivalent to $A$?

Further, we remark that under some conditions on $\Bbbk$, the following inequalities hold.
\begin{align*}
\Theta <_J \Bbbk [x]&/(x^2)<_J A_2\\
\Theta <_J &A_3'<_J A_2.
\end{align*}
However, we do not know whether the dual numbers and $A_3'$ are comparable in the two-sided preorder.
What we can say is that if they are comparable, then $A_3'<_J \Bbbk [x]/(x^2)$. This is because $T_3(A_3')$ is of tame type, whereas $T_3(\Bbbk [x]/(x^2)$ is of finite type.
The dual numbers has underlying quiver
\begin{displaymath}
\xymatrix{
\bullet \ar@(ur,dr)
}
\end{displaymath}
with relation that the arrow twice is zero. In particular, it is minimal with a cycle, so by Proposition~\ref{prop_surj}, any algebra of the form $\Bbbk Q/I$, where the quiver $Q$ has an oriented cycle, will satisfy $\Bbbk Q/I \leq_J \Bbbk [x]/(x^2)$.
Similarly, the Kronecker quiver is minimal with double arrows between distinct nodes, and  $A_3'$ and $(A_3')^\mathrm{op}$ are minimal without double arrow but with sink resp. source of degree strictly greater than 1.
To determine whether or not $\Bbbk [x]/(x^2)$ and $A_3'$ are $J$-comparable would therefore be very interesting.

Finally, we remark that we have no example of a situation where $A\geq_J B$ via bimodules ${}_AM_B$ and ${}_BN_A$ where ${}_AM$ and $N_A$ are non-projective.
Indeed, many of our examples of $J$-relation are even instances of separable extensions, $A\supseteq B$ is a separable extension if and only if $A$ is left-right projective as $B$-$B$-bimodule.
This includes the relation between an algebra and its subalgebra of invariants under a group action (Theorem~\ref{thm_group}), and skew group algebra (Theorem~\ref{thm_skew}), both which are also examples of split extensions.
Moreover, each $A_n$, $n\geq 2$, is a separable extension of the dual numbers, 
which is easily seen by the existence of a separability element (see e.g. \cite{Ka1}).
Also $\Bbbk [x]/(x^2)\geq_J \Theta$, which is not arising from neither a separable nor split extension, is induced by bimoduels which are projective as $\Bbbk [x]/(x^2)$-modules.
This  raises the question whether the one-sided projectivity of the bimodules inducing $J$-relation is a necessary condition.

\bigskip

\begin{center}
\textsc{Acknowledgements}
\end{center}
The author thanks Volodymyr Mazorchuk for many stimulating discussions, and for the original idea behind Theorem~\ref{thm_group}.

\vspace{5mm}

\noindent
Helena Jonsson, Department of Mathematics, Uppsala University, Box. 480,
SE-75106, Uppsala, SWEDEN, email: {\tt helena.jonsson\symbol{64}math.uu.se}


\begin{thebibliography}{99999999}
\bibitem[AF]{AF} F. W.~Anderson,K. R.~Fuller. \textit{Rings and categories of modules.} Graduate texts in mathematics 13, Second edition. Springer, New York, 1998.
%
\bibitem[Br]{Br} M.~Brou\'{e}. Equivalences of blocks of group algebras. V.~Dlab and L.~Scott
eds. Proc. of the NATO Advanced Research Workshop on Representation
of Algebras and Related Topics, Kluwer Academic Publishers, Dordrecht
1991, 1--26.
%
\bibitem[BR]{BR} M.C.R.~Butler, C.M.~Ringel. Auslander-Reiten sequences with few middle terms and
applications to string algebras, Comm. Algebra {\bf 15} (1987), no. 1-2, 145--179.
%
\bibitem[CM]{CM} A.~Chan, V.~Mazorchuk. Diagrams and discrete extensions for finitary $2$-rep\-re\-sen\-ta\-ti\-ons.
Math. Proc. Camb. Phil. Soc. {\bf 166} (2019), 325-352.
%
\bibitem[CS]{CS} X.~Chen, L.~Sun. Singular equivalences of Morita type. Preprint.
%
\bibitem[CB]{CB} W.~Crawley-Boevey. Tame algebras and generic modules. Proc. London. Math. Soc. (3) \textbf{61} (1991), no. 2, 241--265.
%
\bibitem[De]{De} L.~Demonet. Skew group algebras of path algebras and preprojective algebras. J. Alg. 232 \textbf{14} (2010) 1052--1059.
%
\bibitem[Dr]{Dr} Yu. Drozd. Tame and wild matrix problems. Representation theory, II (Proc. Second Internat. Conf., Carleton Univ., Ottawa, Ont., 1979), pp. 242--258, Lecture Notes in Math., {\bf 832}, Springer, Berlin, 1980.
%
\bibitem[DM]{DM} A.~Dugas, R.~Martinez-Villa. A note on stable equivalences of Morita type. J. Pure. Appl. Alg. \textbf{208} (2007) no. 2, 421--433.
%
\bibitem[Gr]{Gr} J.~A.~Green. On the structure of semigroups. Ann. of Math. (2) {\bf 54} (1951), 163--172.  
%
\bibitem[Ha]{Ha} D.~Happel. Hochschild cohomology for finite-dimensional algebras.  In: Malliavin, MP. (eds) Seminaire d'Algebre Paul Dubreil et Marie-Paul Malliavin. Lecture Notes in Mathematics, vol 1404. Springer, Berlin, Heidelberg. https://doi.org/10.1007/BFb0084073
%
\bibitem[HY]{HY} B.~Hou, S.~Yang. Skew group algebras of deformed preprojective algebras. J. Alg. \textbf{332} (2011) 209--228.
%
\bibitem[Jo1]{Jo2} H.~Jonsson. Cell structure of bimodules over radical square zero Nakayama algebras. Comm. Alg. {\bf 48} (2020) no. 2, 573--588.
%
\bibitem[Jo2]{Jo3} H.~Jonsson. On simple transitive 2-representations of bimodules over the dual numbers. Alg. Rep. Th. (2022) https://doi.org/10.1007/s10468-022-10167-w
%
\bibitem[JS]{JS} H.~Jonsson, M~Stroi{\'n}ski. Simple transitive 2-representations of bimodules over radical square zero Nakayama algebras via localization. J. Alg. (2023) \textbf{612}, 87--114.
%
\bibitem[Ka1]{Ka1} L.~Kadison. On split, separable subalgebras with counitality condition. Hokkaido Math. J. \textbf{24} (1995) 527--549.
%
\bibitem[Ka2]{Ka2} L.~Kadison. Separable equivalence of rings and symmetric algebras. Bulletin London Math. Soc. \textbf{52} no. 2 (2019) 344--352.
%
\bibitem[Ka3]{Ka3} L.~Kadison. Uniquely separable extensions. Exp. Math. \textbf{38} (2020), no. 2 217--231.
%
\bibitem[KM]{KM} G.~Kudryavtseva, V.~Mazorchuk. On multisemigroups. Port. Math. {\bf 72} (2015), no. 1, 47--80.
%
\bibitem[Le]{Le} Z.~Leszczy\'{n}ski. On the representation type of tensor algebras. Fund. Math. \textbf{144} (1994) no. 2, 143--161.
%
\bibitem[LS1]{LS1} Z.~Leszczy\'{n}ski, A.~Skowro\'{n}ski. Tame triangular matrix algebras. Colloq. Math. \textbf{86}, no. 2 (2003), 125--145.
%
\bibitem[LS2]{LS2} Z.~Leszczy\'{n}ski, A.~Skowro\'{n}ski. Tame tensor products of algebras. Colloq. Math. \textbf{98}, no. 1 (2000), 259--303.
%
\bibitem[Li]{Li} M.~Linckelmann. Finite generation of Hochschild cohomology of Hecke algebras of finite classical type in characteristic zero. Bull. Lond. Math. Soc. \textbf{43} (2011), no. 5, 871-885.
%
\bibitem[Liu]{Liu} Y.~Liu. Summands of stable equivalences of Morita type. Comm. Alg. \textbf{36} (2008), 3778--3782.
%
\bibitem[MMMTZ]{MMMTZ} M.~Mackaay, V.~Mazorchuk, V.~Miemietz, D.~Tubbenhauer, X.~Zhang. 
Finitary birepresenations of finitary bicategories. Forum Math. {\bf 33} (2021), no. 5, 1261???1320.
%
\bibitem[MM1]{MM1} V. Mazorchuk, V. Miemietz. Cell 2-representations of finitary 2-categories.
Compositio Math \textbf{147}  (2011), 1519--1545.
%
\bibitem[MM2]{MM2} V. Mazorchuk, V. Miemietz. Additive versus abelian 2-representations of 
fiat 2-categories. Mosc. Math. J. {\bf 14} (2014), no. 3, 595--615.
%
\bibitem[MMZ]{MMZ} V.~Mazorchuk, V.~Miemietz, X.~Zhang. Characterisation and applications of
$\Bbbk$-split bimodules. Math. Scand. \textbf{125} no. 2 (2017), 161--178.
%
\bibitem[MZ]{MZ} V.~Mazorchuk, X.~Zhang. Bimodules over uniformly oriented $A_n$ quivers with radical square zero. Kyoto J. Math. \textbf{60}, no. 3 (2020), 965-995.
%
\bibitem[Pe]{Pe} S.F.~Peacock. Separable equivalence, complexity and representation type. J. Alg. \textbf{490} (2017), 219--240.
%
\bibitem[Po1]{Po1} Z.~Pogozarly. Left-right projective bimodules and stable equivalence of Morita type. Colloq. Math. \textbf{88} (2001) no. 2, 243--255.
%
\bibitem[Po2]{Po2} Z.~Pogozarly. Left-sided quasi-invertible bimodules over Nakayama algebras. CEJM. \textbf{3} (2005) no. 1, 125--142.
%
\bibitem[Pr]{Pr} M.~Prest. Purity, spectra and localisation. Encyclopedia of Mathamatics and its Applications, vol. 121, Camebridge University Press, Camebridge (2009).
%
\bibitem[RR]{RR} I.~Reiten, C.~Riedtmann. Skew group algebras in the representation theory of Artin algebras. J. ALg. \textbf{92} (1985), 224-282.
%
\bibitem[Ri]{Ri} J.~Rickard. Derived equivalences as derived functors, J. London Math. Soc. \textbf{43} (1991), 37--48.
%
\bibitem[Ro]{Ro} J.~Rotman. \textit{An introduction to homological algebra.} Pure and applied mathematics 85, Academic press, New York, San Fransisco, London (1979).
%
\bibitem[WW]{WW} B. Wald, J. Waschbusch. Tame biserial algebras. J. Algebra {\bf 95} (1985), no. 2,
480--500.
%
\bibitem[ZZ]{ZZ} G.~Zhou, A.~Zimmermann. On singular equivalences of Morita type. J. Alg.  \textbf{385} (2013), 64--79.
\end{thebibliography}
\end{document}